\pdfoutput=1
\documentclass[11pt]{article}

\usepackage[letterpaper,margin=1in]{geometry}
\usepackage[T1]{fontenc}
\usepackage[utf8]{inputenc}
\usepackage{lmodern}
\usepackage{amsmath,amssymb,amsfonts,amsthm}
\usepackage{graphicx}
\usepackage{subcaption}
\usepackage{booktabs}
\usepackage{enumitem}
\usepackage{microtype}
\usepackage{xcolor}
\usepackage[hidelinks]{hyperref}

\newtheorem{theorem}{Theorem}
\newtheorem{assumption}{Assumption}
\newtheorem{remark}{Remark}

\title{Residual-Certified Adaptive Tracking of Solution Manifolds in Parametric Dynamical Systems}
\author{Yiran Xing\textsuperscript{1} \quad Yuandi Xu\textsuperscript{1} \quad Sulei Hu\textsuperscript{1,*}\\
\small \textsuperscript{1}University of Science and Technology of China, Hefei 230026, China\\
\small \textsuperscript{*}Corresponding author: Sulei Hu\\
\small Specially Appointed Professor, University of Science and Technology of China\\
\small E-mail: husulei@ustc.edu.cn}
\date{\today}

\begin{document}
\maketitle

\begin{abstract}
This paper presents a residual-certified adaptive method for tracking local solution manifolds in parametric dynamical systems. The method combines local POD reduction, full physical residual checks, state-distance snapshot forgetting, high-fidelity resampling, and a lightweight physics-informed neural correction. Instead of learning one global parameter-to-state map, the algorithm maintains the currently active local branch and updates it when the residual indicates loss of validity. The analysis explains why residual thresholds are meaningful on regular branches through local residual-error control, and why stricter local updates are needed near folds or other degenerate neighborhoods. Numerical studies on Ostwald ripening, a particle population-balance model, and the Bratu equation test the approach across low-dimensional dynamics, nonlinear nonlocal residual compensation, and near-fold model failure. The results show that residual-certified local model management can concentrate high-fidelity computation in difficult parameter regions while preserving an interpretable link between surrogate prediction, physical consistency, and active-branch tracking.
\end{abstract}

\noindent\textbf{Keywords:} parametric dynamical systems; solution-manifold tracking; residual certification; model reduction; proper orthogonal decomposition; physics-informed neural networks; snapshot forgetting; Bratu equation

\section{Introduction}

Numerical modeling of complex physical processes is commonly expressed through differential equations, integro-differential equations, or constrained boundary-value problems. The mathematical object of interest is nevertheless broader than a single solve at a fixed parameter value. It is the parametric dynamical system generated by state variables, control parameters, boundary conditions, physical constraints, and evolution operators. For continuum evolution, phase-transition dynamics, multiscale transport, multiphysics coupling, and population-balance dynamics, partial differential equations and integro-differential equations describe local conservation laws, diffusion mechanisms, nonlocal interactions, or equilibrium constraints. The associated computational task is to characterize state trajectories in phase space, stable and unstable manifolds, branch responses under parameter perturbations, folds of solution branches near critical parameters, and local topological reorganizations induced by operator degeneracy. Accordingly, the present work formulates the computational problem as local adaptive tracking of a parametric solution manifold: given control parameters, initial and boundary conditions, and physical constraints, how can one stably, accurately, and interpretably approximate the local state structures generated by parameter variation under finite computational resources?

Under this unified formulation, the Ostwald ripening equation, the particle population-balance model, and the Bratu equation are not merely independent test cases with different physical backgrounds. They form a progressive sequence from regular to complex, from low-dimensional to strongly nonlinear, and from regular to nearly singular. The Ostwald ripening equation represents a low-dimensional or weakly nonlinear evolutionary system, where the main difficulty is to extract a few dominant modes while maintaining physical branch compatibility; its mean-field coarsening origin goes back to Lifshitz--Slyozov--Wagner theory \cite{lifshitz1961kinetics,wagner1961alterung}. The particle population-balance model introduces nonlocal integral coupling and strongly nonlinear feedback, reflecting the sensitive dependence of state evolution on global distributions, conserved moments, and nonnegativity constraints \cite{ramkrishna2000population,marchisio2005solution}. The Bratu equation exhibits a typical parameter bifurcation and a fold of the solution branch; the degeneracy of the Jacobian near the critical point directly challenges conventional reduced-order models, Newton-type iterations, and single-valued neural-operator approximations \cite{joseph1973quasilinear,keller1977numerical}. The experimental system in this paper is therefore a systematic test of the hybrid algorithm across three levels of complexity: low-dimensional nearly linear dynamics, strongly nonlinear nonlocal dynamics, and bifurcating equilibrium solution manifolds.

High-fidelity numerical solvers can provide reliable benchmark solutions at prescribed parameter values. In high-dimensional parameter spaces, large-scale inverse problems, complex branch tracking, and real-time online prediction, however, they face severe computational cost and dimensionality barriers. Reduced-order modeling, especially POD and reduced-basis methods, extracts dominant energy modes from snapshot spaces and can efficiently compress the degrees of freedom when the state manifold is approximately flat and weakly sensitive to parameters \cite{sirovich1987turbulence,berkooz1993proper,rozza2008reduced,benner2015survey}. Such linear projection methods rely on a crucial assumption: the parametric state set can be approximated by a relatively stable low-dimensional linear subspace. When the dynamical system exhibits moving fronts, local high-frequency oscillations, strong nonlocal coupling, coexistence of multiple branches, or bifurcation-critical behavior, the geometry of the state manifold may bend or fold rapidly. A precomputed global basis can then become mismatched with the active branch. This mismatch not only increases truncation error but may also drive the approximate state away from the physically relevant branch and generate nonphysical intermediate states or extrapolated states.

Physics-informed neural networks provide another route for mesh-free modeling of complex dynamical systems. They embed differential operators, boundary conditions, initial conditions, and observational data into a loss functional, and use automatic differentiation and optimization to approximate a function satisfying the physical constraints \cite{raissi2019physics}. Compared with linear reduction, neural networks possess stronger nonlinear expressivity and can in principle compensate high-frequency residuals and local nonlinear structures that are difficult for linear subspaces to capture. Yet, in strongly nonlinear, stiff, or bifurcation-critical regimes, pure neural approaches also encounter structural difficulties. Neural-network training typically has a spectral bias toward low-frequency smooth components and may approximate discontinuities, spikes, moving interfaces, and high-frequency oscillations inefficiently \cite{rahaman2019spectral}. Loss landscapes composed of multiple physical constraints may suffer from gradient competition, scale imbalance, and ill-conditioned optimization, leading to slow convergence, local collapse, or global divergence near unstable manifolds or Jacobian-degenerate points \cite{wang2021understanding}. Moreover, when multiple solution branches correspond to the same parameter, directly learning a single-valued map
\[
\mu\mapsto u(\mu)
\]
is conceptually insufficient, because the true object is a branched parametric solution manifold.

Recent developments in local reduced-order modeling, nonlinear manifold learning, neural operators, residual-adaptive sampling, and bifurcation-neighborhood error estimation have addressed different aspects of these challenges. Local reduced-order models construct multiple local low-dimensional spaces through time windows, parameter partitions, or offline clustering, thereby alleviating the inability of a global basis to cover a curved manifold. Residual-adaptive physics-informed neural networks increase collocation density in high-residual regions, improving the resolution of training samples in physically difficult domains \cite{wu2023comprehensive}. Neural-operator methods attempt to learn the solution operator of a family of parametric equations to achieve fast inference \cite{li2021fourier,kovachki2023neural}. Bifurcation and analytic-error-bound theory reveal that the residual--error relation near folds or higher-order degeneracies may degenerate from linear equivalence to fractional-order control \cite{keller1977numerical,bolte2007clarke}. These methods, however, usually address reduction, neural training, residual sampling, or local singularity detection separately. A unified framework combining local representation, residual certification, snapshot management, and nonlinear compensation remains needed.

The effective computation of complex parametric dynamical systems should therefore not rely on a single model to statically reconstruct a global state manifold. It should instead adopt a local, adaptive, residual-certified tracking paradigm. In regular low-dimensional or weakly nonlinear regions, local POD extracts low-frequency dominant modes at low cost. When the system approaches a strongly nonlinear region, a rapidly curved local manifold, or a bifurcation-critical region, the physical residual senses the deviation between the current reduced space and the true active branch, triggering local snapshot update, basis reconstruction, or high-fidelity resampling. For high-frequency residuals that cannot be expressed by a linear reduced model, a lightweight physics-informed neural network performs residual-space compensation. The algorithm thus focuses on continuously identifying, tracking, and correcting the local solution manifold along evolutionary or equilibrium branches of parametric dynamical systems.

In this framework, the physical residual links dynamical-systems theory to algorithmic adaptivity. For a residual operator
\[
\mathcal R(u;\mu)=0,
\]
the residual norm measures not only the satisfaction of the original physical equations by a candidate state, but also the deviation of the current reduced model from the true active branch. In a locally regular region, a two-sided control between the residual and the state error makes the residual an adaptive triggering signal requiring no additional supervised labels. Near isolated degenerate points associated with fold bifurcations or finite-branch pitchfork bifurcations, the two-sided linear equivalence may fail. The adaptive trigger, however, only requires the state error to be controlled by the residual from one side. When this control degenerates into a fractional-order estimate, stricter residual thresholds and state-distance constraints can still preserve local discriminative power. If the residual remains below the admissible threshold, the algorithm continues with the current local reduced model. If it exceeds the threshold, the algorithm infers that the current state may have left the effective local neighborhood and switches to a new local approximation space through high-fidelity resampling, state-distance forgetting, and basis update.

The proposed residual-certified adaptive hybrid algorithm uses the active local branch as the basic computational unit. It combines the low-dimensional compression capability of POD, the a posteriori certification capability of physical residuals, the snapshot-forgetting mechanism driven by state-space distance, and nonlinear residual compensation through a lightweight neural network. In the online stage, the algorithm no longer uses a fixed global basis generated offline; instead, it dynamically maintains a local snapshot set according to the parameter path and residual feedback. In the inference stage, it first reconstructs the dominant state by a local POD basis and then corrects high-frequency or strongly nonlinear residual components through the neural network. In cross-region evolution, residual thresholds, state distance, and linearized condition numbers jointly determine whether the local model should be updated, thereby reducing online inference cost while maintaining physical compatibility.

The main contributions are as follows.

\begin{enumerate}[label=(\arabic*)]
\item A residual-certified hybrid computational paradigm is proposed for tracking parametric solution manifolds. The computational object is formulated as a parametric solution manifold and its local structural variation under control parameters. Based on this formulation, the method combines local POD, residual triggering, state-distance forgetting, and neural residual compensation. The framework gives a unified interpretation of error sources and model-update mechanisms in low-dimensional nearly linear systems, strongly nonlinear nonlocal systems, and bifurcating equilibrium systems.

\item A residual--error equivalence is established on regular active branches and used as a theoretical basis for algorithmic feasibility. Under local regularity and operator-stability assumptions, and with finite locally distinguishable branches in multibranch neighborhoods, the residual norm controls the distance from a candidate state to the current active branch. This provides a mathematical basis for unsupervised adaptive triggering, local snapshot update, and high-fidelity resampling. The residual is therefore not merely a training loss for neural networks, but an a posteriori certification quantity for local model management and solution-manifold validity.

\item A state-distance snapshot maintenance mechanism and a residual-space neural compensation strategy are developed for multibranch systems. In multiple-solution or branch-switching problems, state-space distance helps preserve the local coverage of the current active branch and reduces the risk that inactive-branch snapshots enter the current POD subspace. For strongly nonlinear and nonlocal problems, the lightweight neural network learns only the residual component after low-dimensional projection,
\[
u_{\mathrm{hyb}}=u_{\mathrm{low}}+g_\theta,
\]
rather than a full solution operator. This compresses the network hypothesis space and mitigates spectral bias and gradient competition.

\item A progressive experimental system consisting of OR, PMC, and Bratu equations is designed to assess the algorithm under different dynamical-system complexities. The OR equation evaluates local manifold compression in low-dimensional nearly linear dynamics; the PMC model tests residual compensation under strongly nonlinear nonlocal feedback; the Bratu equation examines residual discrimination and snapshot update near fold bifurcations, branch degeneracy, and Jacobian near-singularity.
\end{enumerate}

In summary, this paper studies residual-certified adaptive tracking of solution manifolds in parametric dynamical systems. This formulation places stability, branches, bifurcations, manifold evolution, residual control, and model reduction within a common analytical framework. It also makes all algorithmic modules serve a unified objective: local approximation, error certification, and adaptive update of the current active branch. The following sections present the hybrid algorithm, residual--error control and complexity estimates, and numerical validation on OR, PMC, and Bratu systems.

\section{Adaptive Hybrid Algorithm for Parametric Solution-Manifold Tracking}

\subsection{Dynamical-system formulation and computational object}

This paper understands complex physical models as parametric dynamical systems and as local tracking problems for their equilibrium or constrained solution manifolds. The framework contains two related mathematical objects.

The first is a parametric semiflow induced by an evolution equation. Let the state space be \(\mathcal X\) and the parameter space be \(\mathcal P\subset\mathbb R^p\). The system satisfies
\[
\dot u(t)=\mathcal N(u(t);\mu),\qquad u(0)=u_0,\qquad \mu\in\mathcal P .
\]
Under well-posedness conditions it generates a parametric semiflow
\[
\Phi_\mu^t:u_0\mapsto u(t;\mu),
\]
whose objects of study include state trajectories in phase space, stability, local manifold structures, attracting regions, and parameter-perturbation responses. For material evolution, population dynamics, and multiscale transport, this semiflow describes the trajectory structure of physical states under time and control parameters.

The second object is a parametric solution manifold defined by steady equations, boundary-value problems, integral constraint equations, or equilibrium conditions:
\[
\mathcal R(u;\mu)=0,\qquad
\mathcal R:\mathcal X\times\mathcal P\rightarrow\mathcal Y,
\]
where \(\mathcal Y\) is the Banach space of residuals. For evolutionary problems, a dynamical residual may also be defined in a trajectory space
\[
\mathcal X_T\subset L^2(0,T;V)\cap H^1(0,T;V^\ast)
\]
by
\[
\mathcal R_T(u;\mu)=\dot u-\mathcal N(u;\mu).
\]
Thus an entire trajectory can be viewed as the solution of a residual system. To unify notation, the following discussion writes the residual system simply as
\[
\mathcal R(u;\mu)=0,
\]
whether it represents an evolutionary trajectory, a steady solution, a boundary-value solution, or a constrained state.

The parametric solution manifold is
\[
\mathcal M=\{(\mu,u)\in\mathcal P\times\mathcal X:\mathcal R(u;\mu)=0\}.
\]
For a fixed parameter \(\mu\), the corresponding solution fiber is
\[
\mathcal M(\mu)=\{u\in\mathcal X:\mathcal R(u;\mu)=0\}.
\]
The algorithm does not attempt to construct a global approximation of \(\mathcal M\) over the full parameter domain. Instead, along a parameter path it adaptively tracks the local state manifold of the currently active branch. The current active branch, denoted \(\mathcal M_\alpha\subset\mathcal M\), is the local solution set actually visited by the algorithm under the given parameter path and physical constraints. If several solutions exist for the same parameter, the algorithm does not compress all branches into one linear space. It maintains a locally consistent representation of the current physical branch through residual certification, state distance, and local resampling.

The three models used later have a clear progressive relation. The OR equation reflects low-dimensional nearly linear or weakly nonlinear evolution; the PMC model reflects nonlocal integral coupling and strongly nonlinear feedback; the Bratu equation reflects fold bifurcation and Jacobian degeneracy near a critical point. They therefore form a validation chain from low-dimensional nearly linear dynamics to strongly nonlinear nonlocal dynamics and then to bifurcating equilibrium systems.

Two remarks clarify the scope of this formulation. First, the term ``solution manifold'' is used in a local computational sense. It does not require the full set \(\mathcal M\) to be globally smooth, connected, or single-valued over \(\mathcal P\). The algorithm only requires that the portion of \(\mathcal M\) visited by the parameter path can be locally represented, certified, and updated. This distinction is important for systems with folds or multiple branches: a global graph representation \(u=u(\mu)\) may fail, while local fibers and active-branch neighborhoods remain meaningful for computation. Second, the residual system is not introduced merely as a convenient notation. It provides a common quantity that is available for differential, integral, constrained, steady, and trajectory-level formulations. Once a candidate state is generated by a reduced model or a neural correction, the same residual norm can be evaluated in the original physical space and can therefore serve as an a posteriori consistency test across all examples considered here.

\subsection{Spatial setting for local low-dimensional approximation}

POD and Galerkin projection rely on inner products, orthogonal projections, and singular-value decompositions. In addition to Banach-space residual analysis, the paper introduces a Hilbert-space structure. Assume that the state space \(\mathcal X\) is continuously embedded in a Hilbert space \(H\),
\[
\mathcal X\hookrightarrow H.
\]
POD modes, orthogonal projections, and energy truncation are defined in \(H\) or in a finite-dimensional discrete subspace \(H_h\). The residual operator regularity is still analyzed as
\[
\mathcal R:\mathcal X\times\mathcal P\to\mathcal Y .
\]
This distinction is important: POD measures approximation quality in an energy norm, whereas physical residuals may live in weaker or stronger spaces determined by the differential or integral operator.

Let the current local snapshot set be
\[
\mathcal S_k=\{u_{k-m+1},\ldots,u_k\}\subset H_h.
\]
After centering if needed, the snapshot matrix \(U_k\) admits a singular-value decomposition
\[
U_k=W\Sigma V^\top .
\]
Taking the first \(r\) left singular vectors gives the local POD basis
\[
\Phi_k=[\phi_1,\ldots,\phi_r].
\]
The reduced approximation of a candidate state is
\[
u_{\mathrm{low}}(\mu)=\bar u_k+\Phi_k a(\mu),
\]
where \(\bar u_k\) is the local mean and \(a(\mu)\in\mathbb R^r\) is the reduced coordinate vector.

The POD truncation error on the snapshot set satisfies the standard optimality property of POD/SVD snapshot approximation \cite{sirovich1987turbulence,berkooz1993proper}:
\[
\inf_{\dim V_r=r}\left(\sum_{u_i\in\mathcal S_k}\|u_i-P_{V_r}u_i\|_H^2\right)^{1/2}
=\left(\sum_{j>r}\sigma_j^2\right)^{1/2}.
\]
This identity controls only the snapshot set itself. To extend it to a neighborhood of the current parameter, one needs local smoothness of the active branch. If the branch map \(\mu\mapsto u_\alpha^\ast(\mu)\) is locally Lipschitz or \(C^1\) in a parameter neighborhood \(\mathcal P_k\), then for \(\rho_k=\operatorname{diam}(\mathcal P_k)\) one obtains a bound of the form
\[
\operatorname{dist}_H(u_\alpha^\ast(\mu),\operatorname{span}\Phi_k)
\leq C_{\mathrm{POD}}\varepsilon_{\mathrm{POD}}+C_{\mathrm{loc}}\rho_k .
\]
Here \(\varepsilon_{\mathrm{POD}}\) is the snapshot truncation error, and \(C_{\mathrm{loc}}\rho_k\) describes local extrapolation due to parameter-neighborhood radius and manifold curvature. When the parameter path leaves the current neighborhood or the local curvature grows rapidly, POD alone becomes unreliable and residual triggering is required.

In practical computations the local rank \(r\) can be selected by either an energy criterion or a residual-aware criterion. The energy criterion chooses the smallest \(r\) such that
\[
\frac{\sum_{j=1}^r\sigma_j^2}{\sum_{j=1}^{m}\sigma_j^2}\geq 1-\varepsilon_{\mathrm{eng}},
\]
where \(\varepsilon_{\mathrm{eng}}\) is a prescribed truncation tolerance. This rule is effective when the singular-value decay is sharp and the discarded modes have limited physical influence. In strongly nonlinear or nonlocal systems, however, low-energy modes may still produce large residuals because derivatives, integral couplings, or boundary constraints can amplify small state-space errors. For this reason, the algorithm may also enforce a residual-aware check: after constructing \(\Phi_k\), it evaluates the full residual of reconstructed snapshots or validation states and increases \(r\) if the residual is inconsistent with the target threshold. This prevents the local basis from being selected solely by an energy norm that is not aligned with the physical operator.

The local mean \(\bar u_k\) also has a nontrivial role. If the active branch has a slowly drifting baseline, centering the snapshots improves the conditioning of the POD coordinates and prevents the leading mode from merely representing the mean state. For steady or boundary-value problems, centering must be applied consistently with boundary conditions; otherwise the reconstruction may violate homogeneous constraints. In the experiments below, the mean and basis are updated together whenever the snapshot pool is refreshed, so that the affine reduced space remains tied to the current active branch.

\subsection{Low-dimensional dynamical projection and candidate states}

Given a POD basis \(\Phi_k\), a candidate low-dimensional state can be generated by Galerkin, Petrov--Galerkin, or least-residual projection. In Galerkin form one solves
\[
\Phi_k^\top \mathcal R(\bar u_k+\Phi_k a;\mu)=0 .
\]
For non-self-adjoint, stiff, or nonlinearly coupled systems, a least-residual form is often more robust:
\[
a(\mu)=\arg\min_{a\in\mathbb R^r}
\|\mathcal R(\bar u_k+\Phi_k a;\mu)\|_{\mathcal Y_h}^2 .
\]
The resulting state
\[
u_{\mathrm{low}}(\mu)=\bar u_k+\Phi_k a(\mu)
\]
is not immediately accepted as a physical solution. It is treated as a candidate state requiring residual certification.

This distinction is essential. In low-dimensional systems such as OR, the local POD basis may already capture most physically relevant structures, and the residual remains small over several parameter steps. For PMC and Bratu-type problems, the state manifold may show strong curvature, nonlocal feedback, local high-frequency structure, or branch folds. The reduced projection then only provides a candidate state whose credibility must be checked through the full-order residual. The reduced solution is therefore not the final output, but a certifiable candidate generated by the current local model.

The reduced solve can be implemented in several ways depending on the stiffness and algebraic structure of the residual. A Galerkin projection is computationally cheapest when the operator is symmetric or weakly nonsymmetric, but it can become unstable if the reduced test space fails to represent the residual directions that dominate the full model. A Petrov--Galerkin or least-residual formulation is more expensive but aligns the reduced solve with the certification quantity used later. This paper adopts the latter interpretation in the analysis: the candidate should reduce the physical residual within the current low-dimensional affine space, but acceptance is always determined by the full residual evaluated after reconstruction.

For nonlinear residuals, the reduced optimization problem may have multiple local minima. The algorithm therefore uses continuation in parameter space: the reduced coordinate from the previous accepted state initializes the current reduced solve. This is not a global continuation method; it is a numerical warm start that maintains local branch consistency. If the residual after convergence remains too large, the failure is interpreted as evidence that the current reduced space is no longer adequate, not merely as an optimizer failure.

\subsection{Residual triggering compatible with stability estimates}

Let the residual indicator be
\[
\eta(u;\mu)=\|\mathcal R(u;\mu)\|_{\mathcal Y}.
\]
When \(u=u_{\mathrm{low}}(\mu)\), it measures the degree to which the candidate satisfies the original physical equation. The key point is that, in a locally regular region, this residual can also control the distance from the candidate to the true active branch.

Assume that \(D_u\mathcal R(u^\ast_\alpha(\mu);\mu)\) is invertible near the active branch. Then the local inverse theorem gives constants \(c_1,c_2>0\) such that
\[
c_1\|u-u_\alpha^\ast(\mu)\|_{\mathcal X}
\leq
\|\mathcal R(u;\mu)\|_{\mathcal Y}
\leq
c_2\|u-u_\alpha^\ast(\mu)\|_{\mathcal X}.
\]
In multibranch form this can be written as
\[
\operatorname{dist}_{\mathcal X}(u,\mathcal M_\alpha(\mu))
\leq C_r\|\mathcal R(u;\mu)\|_{\mathcal Y}.
\]
Thus the residual trigger is not an empirical threshold alone; it is an a posteriori certification mechanism supported by local stability. The paper uses the residual as a local model-management signal to decide whether the basis should be updated and whether the high-fidelity solver should be called.

The basic rule is
\[
\eta_k=\|\mathcal R(u_{\mathrm{low}}(\mu_k);\mu_k)\|_{\mathcal Y}.
\]
If
\[
\eta_k\leq \tau,
\]
then the current low-dimensional state is accepted and may be further corrected by the neural residual compensator. If
\[
\eta_k>\tau,
\]
the current local basis is considered unreliable; the algorithm calls the high-fidelity solver at \(\mu_k\), obtains a new snapshot, updates the snapshot pool, and reconstructs the POD basis.

The threshold \(\tau\) can be chosen in either absolute or normalized form. An absolute threshold is appropriate when the residual norm has a fixed physical scale. In many parametric problems, however, the residual magnitude varies with the parameter or with the norm of the forcing term. A normalized indicator
\[
\eta_k^{\mathrm{rel}}
=
\frac{\|\mathcal R(u_{\mathrm{low}};\mu_k)\|_{\mathcal Y}}
{1+\|\mathcal F(\mu_k)\|_{\mathcal Y}}
\]
is then more robust, where \(\mathcal F(\mu_k)\) denotes the parameter-dependent forcing or source contribution. The theoretical statements remain unchanged after replacing the residual norm by an equivalent scaled norm on a compact parameter set. In the numerical discussion, the residual threshold should be understood as the threshold applied to a norm compatible with the stability estimate.

It is also useful to distinguish two types of residual growth. Smooth residual growth over several steps indicates ordinary local extrapolation: the active branch is still being followed, but the local affine space is slowly losing accuracy. Abrupt residual growth, especially when accompanied by a sharp increase in the linearized condition number, indicates a possible local structural change, such as a fold neighborhood or a transition to a more nonlinear regime. The first case can often be corrected by a single high-fidelity update, while the second requires a smaller parameter step and stricter state-distance filtering in the snapshot pool.

\subsection{State-distance forgetting and active-branch maintenance}

Residual triggering determines when the local model may have failed. State-distance forgetting determines which snapshots should be retained. If one simply accumulates all historical high-fidelity states, the local snapshot pool may contain old states or snapshots from inactive branches. This can make the POD subspace represent an average of several incompatible structures and create nonphysical intermediate states.

Let the current snapshot pool be \(\mathcal S_k\). A distance-based forgetting rule retains snapshots satisfying
\[
\|u_i-u_{\mathrm{ref}}\|_H\leq R_k ,
\]
where \(u_{\mathrm{ref}}\) is the current accepted state or the latest high-fidelity state, and \(R_k\) is an adaptive radius. Equivalently, one may retain the \(m\) nearest snapshots in \(H\). The updated set is
\[
\mathcal S_{k+1}=\{u_i\in\mathcal S_k\cup\{u_{\mathrm{new}}\}:
\|u_i-u_{\mathrm{ref}}\|_H\leq R_k\}.
\]
This rule is meaningful for multibranch systems because different branches may be close in parameter space but well separated in state space.

Suppose the active branch \(\mathcal M_\alpha\) and another branch \(\mathcal M_\beta\) have a local separation distance
\[
\operatorname{dist}_H(\mathcal M_\alpha,\mathcal M_\beta)
=\inf_{u_\alpha\in\mathcal M_\alpha,\ u_\beta\in\mathcal M_\beta}
\|u_\alpha-u_\beta\|_H
\geq \delta_b,
\]
and the snapshot-pool radius satisfies
\[
\max_{u_i,u_j\in\mathcal S_k}\|u_i-u_j\|_H<\frac{\delta_b}{2}.
\]
Then state-distance maintenance reduces the risk that inactive-branch snapshots remain in the current local basis. The result does not imply that all multibranch systems can be separated by distance alone, but it provides an interpretable basis for local snapshot purification on regular branches.

This is especially important in the Bratu equation. Near a fold, the parameter distance may be small while the state direction changes significantly. Updating snapshots only by parameter distance is therefore insufficient. State distance, residual variation, and local tangent information should be used together to maintain branch consistency.

The forgetting radius \(R_k\) can be fixed, adaptive, or percentile based. A fixed radius is simple but may be inappropriate when the local state scale changes with the parameter. An adaptive radius can be tied to the empirical spread of the current accepted snapshots:
\[
R_k=\kappa
\left(
\frac{1}{|\mathcal S_k|}
\sum_{u_i\in\mathcal S_k}\|u_i-u_{\mathrm{ref}}\|_H^2
\right)^{1/2},
\]
where \(\kappa>0\) controls locality. A nearest-neighbor rule retains a fixed number of snapshots and is often more stable when the sampling density is irregular. In all cases, the key requirement is that the retained set represent a connected local portion of the active branch rather than a mixture of remote historical states. The residual trigger and the forgetting rule therefore play complementary roles: the residual decides when new physical information is needed, while the state distance decides which historical information remains relevant.

The snapshot update can be interpreted geometrically. POD builds the best linear approximation to the current snapshot cloud. If the cloud straddles two branches, its principal directions may point between physically valid states. State-distance forgetting shrinks the cloud around the active branch so that the POD modes approximate local tangent and curvature directions of that branch. This is why the method is particularly useful in multibranch problems even without an explicit global branch-continuation module.

\subsection{One-sided residual control and local update near bifurcation}
\label{subsec:bif-control}

The ordinary residual trigger is linearly justified on regular branches. Near Bratu-type fold bifurcations, however,
\[
D_u\mathcal R(u;\mu)
\]
may lose invertibility, and the residual norm is no longer linearly equivalent to the state error. This paper does not treat complete global branch tracking as the main algorithmic target. Instead, the bifurcation neighborhood is treated as a validity boundary of residual certification. If the active branch is locally separated in state space and the singular point is isolated or removable, the residual can still control the error through a one-sided H\"older-type estimate.

Specifically, let \(\mathcal M_\alpha\) be the current active branch. If, in a neighborhood excluding or isolating the singular point, there exist \(C>0\) and \(0<\alpha\leq 1\) such that
\[
\operatorname{dist}_{\mathcal X}(u,\mathcal M_\alpha(\mu))
\leq C\|\mathcal R(u;\mu)\|_{\mathcal Y}^{\alpha},
\]
then the residual threshold still translates into an error tolerance. Given \(\varepsilon_{\mathrm{tol}}\), it suffices to choose
\[
\tau\leq \left(\frac{\varepsilon_{\mathrm{tol}}}{C}\right)^{1/\alpha}
\]
to guarantee
\[
\operatorname{dist}_{\mathcal X}(u,\mathcal M_\alpha(\mu))
\leq \varepsilon_{\mathrm{tol}}.
\]
The fold normal form
\[
F(x,\lambda)=x^2-\lambda
\]
gives a typical square-root control near the critical point. The finite-branch pitchfork normal form
\[
F(x,\lambda)=x^3-\lambda x
\]
shows that multibranch structure makes the error bound depend on branch separation and local order. Thus residual certification near bifurcation is not completely invalid; it must be interpreted as fractional one-sided control rather than linear equivalence.

Accordingly, the algorithm does not introduce an additional branch-tracking system for Bratu-type problems. It monitors residuals, state distances, and linearized condition numbers in the critical neighborhood. When the residual exceeds the threshold or branch separation is insufficient to support the current local basis, the high-fidelity solver reanchors the active branch, and state-distance forgetting removes snapshots that may belong to other branches. The Bratu equation therefore characterizes the validity boundary of ordinary residual certification and the local update capacity of the basis.

This treatment deliberately separates two tasks that are sometimes conflated. The first task is complete continuation of all branches through a singular point. That task requires additional parametrization information and is not the main objective of the present algorithm. The second task is online certification of whether the current local surrogate remains reliable along the branch being visited. For the second task, a one-sided residual bound and branch-separation condition are sufficient. If the residual is small, the candidate is close to some solution on the active local fiber; if the residual is large, the algorithm does not need to identify the entire branch geometry before acting. It only needs to update the local basis by calling the high-fidelity solver and filtering snapshots. This is the reason the Bratu example can be used to test residual certification near a fold without introducing an explicit pseudo-arclength continuation component.

\subsection{Spectral decoupling and nonlinear residual compensation}

Local POD captures the low-frequency backbone of the state. In PMC-type problems, however, nonlocal strong coupling, boundary-layer effects, high-frequency perturbations, or local spikes may not be expressible by a linear basis. The hybrid approximation is therefore written as
\[
u_{\mathrm{hyb}}=u_{\mathrm{low}}+g_\theta,
\]
where \(g_\theta\) is a lightweight physics-informed neural network representing the residual compensation \cite{raissi2019physics}.

The training objective is
\[
\min_\theta
\left[
\|\mathcal R(u_{\mathrm{low}}+g_\theta;\mu)\|_{\mathcal Y}^2
+\lambda_b\mathcal L_{\mathrm{bc}}
+\lambda_i\mathcal L_{\mathrm{ic}}
+\lambda_r\|g_\theta\|_H^2
\right].
\]
Here \(\mathcal L_{\mathrm{bc}}\) and \(\mathcal L_{\mathrm{ic}}\) denote boundary- and initial-condition errors, and \(\lambda_b,\lambda_i,\lambda_r\) are weights. The regularization term \(\lambda_r\|g_\theta\|_H^2\) suppresses nonphysical oscillations produced by the compensator.

The key design is that the neural network does not approximate the full field. It learns only the local residual after POD truncation. If the exact state is decomposed as
\[
u^\ast=u_{\mathrm{low}}^\ast+u_{\mathrm{res}}^\ast,
\]
then the network seeks
\[
g_\theta\approx u_{\mathrm{res}}^\ast .
\]
This greatly reduces the hypothesis space and is intended to mitigate spectral bias and ill-conditioned multi-term optimization, two documented difficulties of neural PDE solvers \cite{rahaman2019spectral,wang2021understanding}. For OR, \(g_\theta\) usually provides only a small correction. For PMC, it is important for compensating nonlocal strong-coupling error. For Bratu, it mainly improves local accuracy and cannot replace residual certification and high-fidelity resampling for branch discrimination.

The regularization on \(g_\theta\) is not merely a numerical convenience. Since the low-dimensional component already carries the dominant energy, an unconstrained neural correction could overfit residual collocation points and introduce oscillatory artifacts that reduce the training loss but degrade the state norm or violate physical shape constraints. The penalty \(\|g_\theta\|_H^2\), together with boundary and initial losses, biases the correction toward the smallest physically admissible modification of the POD state. In distribution-valued problems such as OR and PMC, additional weak penalties on mass conservation or nonnegativity can be included without changing the framework. The compensator should be understood as a local corrector, not as a replacement for the high-fidelity solver.

\subsection{Closed-loop parameter scanning and algorithmic workflow}

Let the outer task be driven by an objective functional
\[
J(\mu)=\mathcal G(u(\mu),\mu),
\]
which may represent observation misfit, performance index, critical-state criterion, or inverse-problem loss. An outer scanner provides a sequence
\[
\{\mu_k\}_{k=0}^K,
\]
and the inner hybrid algorithm returns state approximations, residual indicators, and local-model credibility.

The workflow is as follows.
\begin{enumerate}[label=\arabic*.]
\item Call the high-fidelity solver in an initial parameter neighborhood to generate the initial snapshot set \(\mathcal S_0\).
\item Construct the local POD basis \(\Phi_0\) in \(H_h\).
\item At the current parameter \(\mu_k\), solve for the candidate low-dimensional state \(u_{\mathrm{low}}(\mu_k)\) by Galerkin, Petrov--Galerkin, or least-residual projection.
\item Compute the residual indicator
\[
\eta_k=\|\mathcal R(u_{\mathrm{low}};\mu_k)\|_{\mathcal Y}.
\]
\item If \(\eta_k\leq\tau\), accept the low-dimensional state and train or update the compensator \(g_\theta\) to obtain \(u_{\mathrm{hyb}}=u_{\mathrm{low}}+g_\theta\).
\item If \(\eta_k>\tau\), call the high-fidelity solver to generate \(u_{\mathrm{new}}\), perform state-distance forgetting, and update the snapshot set and POD basis.
\item If the condition number of \(D_u\mathcal R\) deteriorates, the state distance grows rapidly, or the system approaches a Bratu-type fold neighborhood, reduce the parameter step, tighten the residual criterion, and preferentially reanchor the active branch by high-fidelity computation.
\item Feed \(u_{\mathrm{hyb}}\), residual indicators, and local-branch information back to the outer functional \(J(\mu)\), which selects the next parameter point.
\end{enumerate}

Figure~\ref{fig:workflow} summarizes this closed-loop workflow.

\begin{figure}[htbp]
\centering
\includegraphics[width=0.96\textwidth]{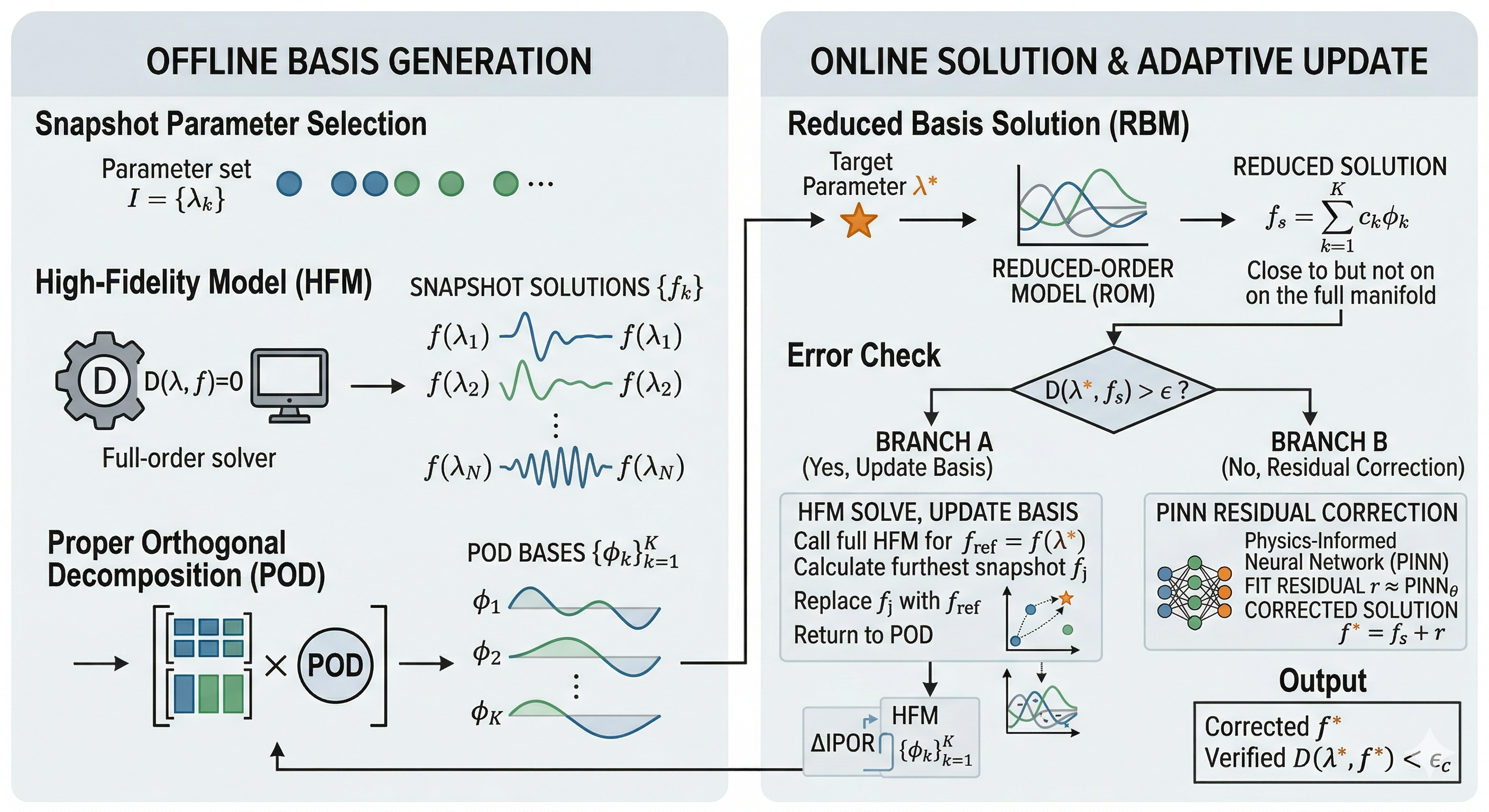}
\caption{Workflow of the residual-certified adaptive hybrid algorithm. The parameter scan generates candidate states by a local POD model; the physical residual certifies local validity; high-fidelity resampling and state-distance forgetting update the active snapshot basis; and a lightweight neural compensator corrects nonlinear residual components.}
\label{fig:workflow}
\end{figure}

The process is a closed loop. Parameter advancement changes the system state; state variation shifts the local manifold; residuals detect manifold deviation; high-fidelity resampling corrects the local basis; neural compensation reduces linear-projection error; and stricter residual criteria plus snapshot updates preserve local consistency near bifurcation. Computational resources are concentrated in regions where nonlinear effects or local structural changes actually occur.

For clarity, the operational logic can be summarized as a three-level decision structure. The first level is prediction: the local reduced model predicts a candidate state using the current POD basis. The second level is certification: the full residual and, when available, the linearized condition number evaluate whether the candidate lies within the certified neighborhood of the active branch. The third level is repair: if certification fails, high-fidelity resampling and state-distance forgetting reconstruct the local approximation space. This structure ensures that the neural compensator is only used after the low-dimensional state has passed a residual-based credibility check, and that high-fidelity computation is reserved for points where the existing local model has demonstrably lost validity.

\subsection{Summary of the algorithmic section}

This section presented the adaptive hybrid architecture for tracking parametric solution manifolds. It distinguished evolutionary semiflows from equilibrium solution manifolds; clarified that POD is defined in Hilbert or discrete inner-product spaces while residual regularity is analyzed in Banach spaces; constructed the low-frequency backbone by local POD; used compatible residual norms to judge deviation of the reduced model from the true branch; maintained the active-branch snapshot set by state-distance forgetting; and compensated high-frequency nonlinear residuals with a lightweight physics-informed neural network.

For OR-type nearly linear systems, the algorithm mainly performs local low-dimensional manifold compression. For PMC-type strongly nonlinear systems, spectral decoupling and residual compensation become more important. For Bratu-type bifurcating systems, residual triggering must be interpreted together with state distance, condition-number monitoring, and high-fidelity resampling. This structure underlies the residual--error control, complexity estimates, and global error decomposition developed next.

\section{Feasibility, Error Control, and Complexity Analysis}

\subsection{Parametric solution manifolds and regular-branch assumptions}

Let \(\mathcal X\) be the state space, \(\mathcal Y\) the residual space, and \(\mathcal P\subset\mathbb R^p\) a compact parameter domain. The residual operator is
\[
\mathcal R:\mathcal X\times\mathcal P\rightarrow\mathcal Y,\qquad
\mathcal R(u;\mu)=0.
\]
The solution manifold and solution fiber are
\[
\mathcal M=\{(\mu,u)\in\mathcal P\times\mathcal X:\mathcal R(u;\mu)=0\},
\qquad
\mathcal M(\mu)=\{u\in\mathcal X:\mathcal R(u;\mu)=0\}.
\]
If \((\mu_0,u_0)\in\mathcal M\) and
\[
D_u\mathcal R(u_0;\mu_0):\mathcal X\rightarrow\mathcal Y
\]
is a Banach-space isomorphism, then \((\mu_0,u_0)\) is called a regular point. Otherwise it is singular. The singular set is denoted by \(\Sigma\). The analysis focuses on local regular branches outside \(\Sigma\) and on local branches near isolated singular points for which fractional one-sided error bounds are available.

The following assumptions are used.

\begin{assumption}[Operator regularity]
The residual operator satisfies
\[
\mathcal R\in C^1(\mathcal X\times\mathcal P,\mathcal Y),
\]
and \(\mathcal M\neq\varnothing\) in the parameter domain considered.
\end{assumption}

\begin{assumption}[Local branch structure]
After removing \(\Sigma\), the solution manifold decomposes into finitely many local \(C^1\) branches:
\[
\mathcal M\setminus\Sigma=\bigcup_{\alpha\in A}\mathcal M_\alpha .
\]
At any time, the algorithm tracks one active branch \(\mathcal M_\alpha\).
\end{assumption}

\begin{assumption}[Local compactness]
The state set actually visited by the algorithm lies in a locally compact set \(K\subset\mathcal X\). This excludes unbounded oscillations, numerical blow-up, and escape into nonphysical states.
\end{assumption}

\begin{assumption}[Local branch distinguishability]
If two branches \(\mathcal M_\alpha\) and \(\mathcal M_\beta\) coexist in the same parameter neighborhood, then in noncritical regions there exists \(\delta_b>0\) such that
\[
\inf_{u_\alpha\in\mathcal M_\alpha(\mu),\,u_\beta\in\mathcal M_\beta(\mu)}
\|u_\alpha-u_\beta\|_{\mathcal X}\geq \delta_b .
\]
\end{assumption}

The assumptions do not imply global uniqueness or unconditional stability near all singularities. They require local invertibility on regular branches, finite branch coverage, and local geometric controllability through a one-sided residual bound near isolated bifurcations.

Assumption 4 is intentionally local. In nonlinear equilibrium problems, two branches can approach each other in parameter space and may even meet at a critical point. The algorithm does not require a uniform global separation over the entire parameter domain. It requires only that the snapshot pool used to construct the current reduced basis not contain incompatible states from different branches. This is enforced computationally by state-distance filtering and residual thresholds. If the separation becomes too small to maintain a branch-consistent snapshot pool, residual smallness is no longer interpreted as evidence of global correctness; instead, residual growth and loss of separation indicate that the local representation must be rebuilt with new high-fidelity information.

The compactness assumption has a similar computational interpretation. Residual norms alone may be insufficient if numerical iterates leave the physically admissible set and enter regions where the operator behaves pathologically. Restricting attention to a locally compact set \(K\) represents the combined effect of boundary conditions, conservation constraints, nonnegativity constraints, regularization, and solver safeguards. Within such a set, local constants in the estimates can be chosen uniformly over the region actually visited by the algorithm.

\subsection{Residual--error equivalence on regular local branches}

The residual certification mechanism rests on the fact that, near a regular branch, the residual norm controls the distance from a candidate state to the true active branch. Because a multibranch system may have several solutions at the same parameter, we write the local solution on the active branch as \(u_\alpha^\ast(\mu)\), not as a global solution \(u^\ast(\mu)\).

\begin{theorem}[Residual--error equivalence on a regular branch]
Let \((\mu_0,u_0)\in\mathcal M_\alpha\) be a regular point on the current active branch, and assume that
\[
D_u\mathcal R(u_0;\mu_0):\mathcal X\rightarrow\mathcal Y
\]
is a Banach-space isomorphism. Then there exist a parameter neighborhood \(U\subset\mathcal P\), a state neighborhood \(V\subset\mathcal X\), a unique local branch map
\[
u_\alpha^\ast:U\rightarrow V,\qquad
\mathcal R(u_\alpha^\ast(\mu);\mu)=0,
\]
and constants \(0<c_1\leq c_2<\infty\) such that for all \(\mu\in U\) and \(u\in V\),
\[
c_1\|u-u_\alpha^\ast(\mu)\|_{\mathcal X}
\leq
\|\mathcal R(u;\mu)\|_{\mathcal Y}
\leq
c_2\|u-u_\alpha^\ast(\mu)\|_{\mathcal X}.
\]
\end{theorem}

\begin{proof}
Since \(D_u\mathcal R(u_0;\mu_0)\) is an isomorphism and \(\mathcal R\in C^1\), the implicit-function theorem gives neighborhoods \(U\) and \(V\) in which the active branch is represented by a unique \(C^1\) map \(u_\alpha^\ast(\mu)\). The local inverse-mapping theorem then implies that, for each \(\mu\in U\), the map \(u\mapsto\mathcal R(u;\mu)\) is locally bi-Lipschitz on \(V\), with constants uniform if the neighborhood is chosen small enough. Hence
\[
\|u-u_\alpha^\ast(\mu)\|_{\mathcal X}
\leq
C\|\mathcal R(u;\mu)-\mathcal R(u_\alpha^\ast(\mu);\mu)\|_{\mathcal Y}.
\]
Because \(\mathcal R(u_\alpha^\ast(\mu);\mu)=0\), the lower bound follows with \(c_1=C^{-1}\). The upper bound follows from the local boundedness of \(D_u\mathcal R\) and the Banach-space integral mean-value formula:
\[
\mathcal R(u;\mu)-\mathcal R(u_\alpha^\ast(\mu);\mu)
=\int_0^1
D_u\mathcal R(u_\alpha^\ast(\mu)+s(u-u_\alpha^\ast(\mu));\mu)
(u-u_\alpha^\ast(\mu))\,ds .
\]
Taking norms gives the desired upper inequality.
\end{proof}

Thus, in a regular neighborhood, the physical residual is an equivalent measure of the distance from a candidate state to the active branch. If
\[
\eta(u;\mu)=\|\mathcal R(u;\mu)\|_{\mathcal Y}
\]
and \(\eta(u;\mu)\leq\tau\), then
\[
\|u-u_\alpha^\ast(\mu)\|_{\mathcal X}\leq c_1^{-1}\tau .
\]
This local residual certification is consistent with the residual-based a posteriori viewpoint used in certified reduced-basis methods, although the present work uses it for online local model management rather than only for offline error certification \cite{rozza2008reduced,benner2015survey}. For multibranch systems it is often more appropriate to write
\[
\operatorname{dist}_{\mathcal X}(u,\mathcal M_\alpha(\mu))
\leq C_r\|\mathcal R(u;\mu)\|_{\mathcal Y},
\]
where
\[
\operatorname{dist}_{\mathcal X}(u,\mathcal M_\alpha(\mu))
=\inf_{v\in\mathcal M_\alpha(\mu)}\|u-v\|_{\mathcal X}.
\]
This form does not require a globally unique solution; it requires only local distinguishability and regularity of the active branch.

The theorem also explains why a small residual should not be interpreted without specifying the branch neighborhood. If two exact solutions coexist at the same parameter and are both close to the candidate in residual norm, the residual alone identifies proximity to the solution set but not necessarily to the desired physical branch. The active-branch notation \(\mathcal M_\alpha\) and the state-distance filtering mechanism are therefore required to turn residual smallness into branch-consistent model management. In regular single-branch regions this distinction disappears; in multibranch regions it is central.

There is also an asymmetry between acceptance and update. To accept a candidate, one needs an upper bound on the state error in terms of the residual. To trigger an update, it is enough to observe that the residual has exceeded the certified tolerance. The algorithm does not need to know the exact state error at the moment of failure. This is important because high-fidelity reference states are expensive and are called only after the residual indicates that they are needed.

\subsection{One-sided residual control near bifurcation}

Theorem 1 applies only to regular branches. Near Bratu-type folds,
\[
D_u\mathcal R(u;\mu)
\]
may be noninvertible, and the ordinary residual--error relation may cease to be linearly two-sided. Nevertheless, for adaptivity it is enough that the error be controlled by the residual from one side.

Let \(\mathcal Z_\mu=\{u:\mathcal R(u;\mu)=0\}\). Suppose that near the active branch there exist \(C>0\), \(0<\alpha\leq 1\), and \(r>0\) such that, whenever \(\|u-u_c\|_{\mathcal X}<r\) and \(\|\mu-\mu_c\|<r\),
\[
\operatorname{dist}_{\mathcal X}(u,\mathcal Z_\mu)
\leq
C\|\mathcal R(u;\mu)\|_{\mathcal Y}^{\alpha}.
\]
Then
\[
\|\mathcal R(u;\mu)\|_{\mathcal Y}\leq\tau
\quad\Longrightarrow\quad
\operatorname{dist}_{\mathcal X}(u,\mathcal Z_\mu)
\leq C\tau^\alpha .
\]
When \(\alpha=1\), this is the regular Lipschitz-type bound. When \(0<\alpha<1\), the residual must be reduced more aggressively to achieve the same state-error tolerance. For a prescribed \(\varepsilon_{\mathrm{tol}}\), the threshold may be chosen as
\[
\tau\leq\left(\frac{\varepsilon_{\mathrm{tol}}}{C}\right)^{1/\alpha}.
\]

The fold normal form \(F(x,\lambda)=x^2-\lambda\) illustrates square-root control at the critical point, and it is the same local degeneracy that motivates continuation methods near folds \cite{keller1977numerical}. The finite-branch pitchfork normal form \(F(x,\lambda)=x^3-\lambda x\) shows that multibranch cases require branch separation or active-branch selection. More generally, H\"older-type residual-to-distance estimates are compatible with analytic and tame error-bound theory such as the Lojasiewicz and Kurdyka--Lojasiewicz frameworks \cite{bolte2007clarke}. If the number of branches is finite, inactive branches remain locally separated from the current snapshot pool, and the singular points do not form a continuous degenerate set, the residual remains a useful indicator of local model failure.

\begin{remark}[Fold and pitchfork normal forms]
The fractional exponent \(\alpha\) can be interpreted as a local regularity index of the residual map near the degenerate point. For the fold normal form at \(\lambda=0\), \(F(x,0)=x^2\), so \(|x|\leq |F(x,0)|^{1/2}\). The residual therefore controls the error only with exponent \(1/2\). For a pitchfork-type normal form at the critical parameter, \(F(x,0)=x^3\), a cubic degeneracy gives an exponent \(1/3\) if all local branches are considered together. If the active branch is restricted by additional physical constraints or by a one-sided parameter path, the effective exponent may improve on that restricted neighborhood. This illustrates why finite-branch structure and active-branch selection matter in residual certification.
\end{remark}

\begin{remark}[Threshold interpretation near degeneracy]
When \(\alpha<1\), the same residual tolerance corresponds to a larger state-error tolerance than in the regular case. A threshold that is adequate on a regular branch may therefore be too loose near a degenerate point. The algorithmic response is not to assume linear equivalence, but to shrink the admissible residual level, reduce the parameter step, and refresh the high-fidelity anchor more frequently. This is precisely the behavior expected in the Bratu experiment.
\end{remark}

\subsection{Feasibility conditions for local POD approximation}

Let the current snapshot set be
\[
\mathcal S_k=\{u_1,\ldots,u_m\},
\]
and the POD basis be
\[
\Phi_k=[\phi_1,\ldots,\phi_r].
\]
For existing snapshots, the POD truncation error satisfies
\[
\left(\sum_{j>r}\sigma_j^2\right)^{1/2}
=\varepsilon_{\mathrm{POD}}.
\]
This error controls the snapshots directly, not the entire parameter neighborhood. If the active branch map \(\mu\mapsto u_\alpha^\ast(\mu)\) is locally Lipschitz or \(C^1\) on \(\mathcal P_k\), then for \(\rho_k=\operatorname{diam}(\mathcal P_k)\) one has
\[
\operatorname{dist}_H(u_\alpha^\ast(\mu),\operatorname{span}\Phi_k)
\leq C_p\varepsilon_{\mathrm{POD}}+C_{\mathrm{loc}}\rho_k .
\]
The first term is spectral truncation error, and the second is local extrapolation or curvature error. Thus local POD is effective when the local curvature is small, the parameter step is short, and snapshot coverage is sufficient. When the curvature increases sharply or the system enters a bifurcation neighborhood, the term \(C_{\mathrm{loc}}\rho_k\) can grow quickly, causing local-basis failure.

This estimate separates statistical and geometric errors. The POD truncation term depends on the singular-value decay of the available snapshots and can be reduced by increasing the rank or improving local snapshot diversity. The curvature term depends on how far the current parameter has moved from the region represented by the snapshots. It cannot be eliminated by increasing the rank alone if the snapshots do not cover the relevant portion of the active branch. Residual triggering is therefore necessary even when the POD energy truncation appears small: the reduced space may approximate old snapshots accurately while failing to approximate the current state.

\subsection{Snapshot maintenance with branch consistency}

The goal of snapshot maintenance is to construct a local basis associated with the current active branch, not to preserve all historical data. Let the updated high-fidelity state be \(u_{\mathrm{new}}\). The snapshot update is
\[
\mathcal S_{k+1}=
\operatorname{Forget}\bigl(\mathcal S_k\cup\{u_{\mathrm{new}}\};u_{\mathrm{new}},R_k\bigr),
\]
where
\[
\operatorname{Forget}(\mathcal A;u_{\mathrm{ref}},R)
=\{u\in\mathcal A:\|u-u_{\mathrm{ref}}\|_H\leq R\}.
\]
Alternatively, the closest \(m\) snapshots may be retained. If branches are locally separated by \(\delta_b\) and \(R_k<\delta_b/2\), snapshots from other branches are excluded from the local pool. For Bratu-type bifurcation systems, state distance is more informative than parameter distance because different branches may correspond to nearly identical parameter values but have distinct local directions in state space.

The forgetting mechanism introduces a controlled loss of historical information. This loss is beneficial when old states lie outside the current local validity region, but harmful if the pool becomes too small or loses important tangent directions. In implementation, the snapshot pool should maintain a minimal size \(m_{\min}\) and may retain a few directionally informative states even if they are not the nearest in norm. A practical rule is to combine nearest-neighbor retention with a rank check on the local snapshot matrix. If the smallest retained singular values collapse, the algorithm can temporarily enlarge the radius or request additional high-fidelity samples. This prevents the local POD basis from becoming underdetermined after aggressive forgetting.

Geometrically, POD builds the best linear approximation to the current snapshot cloud. If the cloud straddles two branches, its principal directions may point between physically valid states. State-distance forgetting shrinks the cloud around the active branch so that the POD modes approximate local tangent and curvature directions of that branch. This is why the method is useful in multibranch problems even without an explicit global branch-continuation module.

\subsection{Single-step feasibility chain}

At parameter \(\mu_k\), suppose the local POD approximation yields \(u_{\mathrm{low}}\). If
\[
\eta_k=\|\mathcal R(u_{\mathrm{low}};\mu_k)\|_{\mathcal Y}\leq\tau,
\]
then by residual--error control,
\[
\|u_{\mathrm{low}}-u_\alpha^\ast(\mu_k)\|_{\mathcal X}\leq c_1^{-1}\tau
\]
in a regular neighborhood. The candidate is therefore acceptable under the current threshold.

If \(\eta_k>\tau\), the current low-dimensional space may no longer cover the true state branch. The algorithm calls the high-fidelity solver to obtain \(u_h(\mu_k)\), updates the snapshots, and reconstructs the POD basis. If the linearized operator is ill-conditioned or the path enters a fold neighborhood, the parameter step is reduced and the branch-separation criterion in snapshot forgetting is strengthened.

Finally, the compensator is introduced on top of the low-dimensional backbone:
\[
u_{\mathrm{hyb}}=u_{\mathrm{low}}+g_\theta.
\]
The network minimizes
\[
\|\mathcal R(u_{\mathrm{low}}+g_\theta;\mu_k)\|_{\mathcal Y}^2
\]
together with initial and boundary losses, thereby compensating high-frequency and strongly nonlinear residual components. Since POD captures the low-frequency backbone, \(g_\theta\) only approximates the remaining residual component and is more stable than a full-field PINN approximation.

This chain also explains why residual certification should precede neural correction. If the POD candidate is far from the active branch, a neural network trained only to reduce residual collocation loss may move the state toward a different branch or a nonphysical local minimizer. By first checking whether the POD candidate lies in a certified neighborhood, the method confines neural compensation to small residual components for which local correction is meaningful. When the candidate fails certification, the correct response is not to increase neural-network capacity, but to refresh the local physical information.

\subsection{Deterministic residual growth and update interval}

Let the parameter step be
\[
\Delta\mu_k=\mu_{k+1}-\mu_k.
\]
In a regular neighborhood, if \(\mathcal R\) is locally Lipschitz in \(u\) and \(\mu\), and local POD extrapolation is controlled, then there exist constants \(C_\mu,C_p,C_\theta>0\) such that
\[
\eta_{k+1}\leq
\eta_k+C_\mu\|\Delta\mu_k\|
+C_p\varepsilon_{\mathrm{POD}}
+C_\theta\varepsilon_\theta
+\mathcal O(\|\Delta\mu_k\|^2),
\]
where \(\varepsilon_\theta\) denotes the residual error not captured by the neural compensator. Residual growth arises from parameter-induced manifold drift, POD truncation and extrapolation, and neural compensation error.

For a fixed step \(\|\Delta\mu_k\|=\Delta\mu\), set
\[
\delta_\eta=C_\mu\Delta\mu+C_p\varepsilon_{\mathrm{POD}}+C_\theta\varepsilon_\theta.
\]
The number of consecutive steps \(m\) between two high-fidelity updates approximately satisfies
\[
m\delta_\eta\lesssim\tau-\eta_0,
\]
so that
\[
m_{\mathrm{upd}}\approx
\frac{\tau-\eta_0}
{C_\mu\Delta\mu+C_p\varepsilon_{\mathrm{POD}}+C_\theta\varepsilon_\theta}.
\]
Reducing the parameter step, improving POD truncation, and reducing neural compensation error all lengthen the valid advance distance between high-fidelity calls. In strongly nonlinear or bifurcating neighborhoods, the residual growth rate increases and the update interval shortens.

The estimate is deliberately deterministic rather than probabilistic. It does not assume a distribution over parameters or residual increments. Instead, it identifies the quantities that control update frequency along a realized parameter path. This is appropriate for inverse problems and continuation-like scans where the path is generated adaptively by an outer optimizer. If the optimizer takes large steps in a region of high curvature, \(\|\Delta\mu_k\|\) increases residual growth and triggers more high-fidelity calls. If the optimizer approaches a smooth valley, residual growth slows and the reduced model can be reused for more steps.

\subsection{Computational complexity}

Let \(T_{\mathrm{ROM}}\) be the cost of one reduced solve, \(T_{\mathrm{NN}}\) the cost of training or updating the compensator, and \(T_{\mathrm{HFM}}\) the cost of one high-fidelity resampling. If there are \(K\) parameter steps and \(N_{\mathrm{upd}}\) high-fidelity updates, the total cost is
\[
T_{\mathrm{total}}
=K(T_{\mathrm{ROM}}+T_{\mathrm{NN}})
+N_{\mathrm{upd}}T_{\mathrm{HFM}}.
\]
The update-interval estimate gives
\[
N_{\mathrm{upd}}\approx \frac{K}{m_{\mathrm{upd}}}.
\]
The advantage of the algorithm arises when
\[
N_{\mathrm{upd}}\ll K,
\]
meaning that the high-fidelity solver is called only when the residual exceeds the threshold, the local manifold curves rapidly, or the branch approaches a critical region. For OR, \(N_{\mathrm{upd}}\) is usually small. For PMC, strong nonlinearity and nonlocal coupling increase \(N_{\mathrm{upd}}\), but neural compensation can reduce frequent resampling. For Bratu, residual growth and condition-number deterioration near the critical parameter shorten the update interval and increase high-fidelity calls.

This cost model highlights that the method is not designed to eliminate high-fidelity computation. Rather, it redistributes high-fidelity effort toward informative points. A static offline reduced model pays a large cost before knowing which regions of the parameter space will be visited. The present algorithm pays high-fidelity cost online only when residual certification indicates that the current local representation has become inadequate. This is especially advantageous in inverse problems, where the optimizer may explore only a small subset of the nominal parameter domain.

The neural compensation cost \(T_{\mathrm{NN}}\) may be interpreted in two ways. If the network is retrained from scratch at every step, \(T_{\mathrm{NN}}\) can dominate the reduced solve. In the intended use, the compensator is warm-started and updated locally, so that it behaves as a small corrective model rather than a full surrogate. Residual-triggered high-fidelity updates also improve neural training by supplying fresh local information, reducing the burden on the network to extrapolate.

\subsection{Global error decomposition}

Let \(u^\ast\) be the exact solution, \(u_h\) the high-fidelity discrete solution, \(u_{\mathrm{low}}\) the POD solution, and
\[
u_{\mathrm{hyb}}=u_{\mathrm{low}}+g_\theta.
\]
In a regular branch neighborhood,
\[
u_{\mathrm{hyb}}-u^\ast
=(u_h-u^\ast)
+(u_{\mathrm{low}}-u_h)
+(g_\theta-u_{\mathrm{res}}^\ast)
+e_{\mathrm{trig}},
\]
where the terms represent high-fidelity discretization error, POD projection and local extrapolation error, neural residual compensation error, and residual-threshold induced update-lag error.

If
\[
\|u_h-u^\ast\|_{\mathcal X}\leq\varepsilon_h,\qquad
\|u_{\mathrm{low}}-u_h\|_{\mathcal X}\leq
C_p\varepsilon_{\mathrm{POD}}+C_{\mathrm{loc}}\rho_k,
\]
\[
\|g_\theta-u_{\mathrm{res}}^\ast\|_{\mathcal X}\leq\varepsilon_\theta,
\qquad
\|e_{\mathrm{trig}}\|_{\mathcal X}\leq C_r\tau,
\]
then
\[
\|u_{\mathrm{hyb}}-u^\ast\|_{\mathcal X}
\leq
C_1\varepsilon_h+C_2\varepsilon_{\mathrm{POD}}
+C_3\rho_k+C_4\varepsilon_\theta+C_5\tau.
\]
For evolutionary systems on \([0,T]\), a Lipschitz stability assumption adds a Gronwall factor:
\[
\sup_{0\leq t\leq T}
\|u_{\mathrm{hyb}}(t)-u^\ast(t)\|_{\mathcal X}
\leq
e^{LT}
\left(
C_1\varepsilon_h+C_2\varepsilon_{\mathrm{POD}}
+C_3\rho_k+C_4\varepsilon_\theta+C_5\tau
\right).
\]
In Bratu-type bifurcation neighborhoods, one should not use the regular-branch linear bound. A fractional one-sided estimate is more appropriate:
\[
\operatorname{dist}_{\mathcal X}(u,\mathcal Z_\mu)
\leq
C\|\mathcal R(u;\mu)\|_{\mathcal Y}^{\alpha},
\qquad 0<\alpha\leq 1.
\]
This avoids incorrectly applying linear residual--error equivalence at a degenerate \(D_u\mathcal R\), and it explains why residual thresholds must be chosen more cautiously near bifurcation.

The error decomposition is additive in presentation, but the terms are not completely independent. A poor POD basis can increase the difficulty of neural compensation, thereby increasing \(\varepsilon_\theta\). A loose residual threshold can delay high-fidelity updates, increasing \(\rho_k\) and the local extrapolation error. Conversely, frequent high-fidelity updates reduce \(\rho_k\) and improve the POD basis, but they increase computational cost. The algorithm balances accuracy and efficiency through the residual threshold, snapshot radius, reduced rank, and compensation strength. These quantities should be tuned together rather than in isolation.

The decomposition also clarifies the role of the high-fidelity solver. The term \(\varepsilon_h\) is the irreducible error of the reference discretization. The hybrid model is not expected to be more accurate than the high-fidelity data on which it is anchored unless additional regularization or data assimilation is used. Its goal is to approach high-fidelity accuracy over many parameter steps at a lower average cost, while detecting when high-fidelity reanchoring is necessary.

\subsection{Theoretical boundaries and scope}

The theoretical scope consists of three regimes. First, in regular low-dimensional regions, \(D_u\mathcal R\) is invertible, the state-manifold curvature is small, and local POD captures the dominant structure. OR mainly tests this regime. Second, in regular but strongly nonlinear regions, \(D_u\mathcal R\) remains locally invertible, but curvature grows and POD extrapolation error increases. PMC belongs to this class, where nonlocal coupling makes neural residual compensation important. Third, in singular or nearly singular regions, the condition number of \(D_u\mathcal R\) deteriorates and linear residual--error equivalence fails. The fold point of the Bratu equation belongs to this class; the residual should be interpreted as a local failure indicator under fractional one-sided control and used together with state distance, condition-number monitoring, and high-fidelity resampling.

Thus the paper does not claim unconditional global convergence for arbitrary nonlinear dynamical systems. It establishes residual equivalence on regular branches and uses one-sided residual control near isolated fold or finite-branch bifurcations under additional branch-separation assumptions. Under these conditions, local POD, residual triggering, high-fidelity update, and neural compensation form an interpretable adaptive solution-manifold tracking algorithm.

\section{Numerical Experiments and Algorithmic Assessment}

\subsection{Experimental objective and progressive validation}

To assess the residual-certified adaptive hybrid algorithm across different dynamical-system complexities, this section considers the Ostwald ripening equation, a particle population-balance model, and the one-dimensional Bratu equation. These examples form a progressive validation chain from regular to singular, from low-dimensional to strongly nonlinear, and from single-branch to multibranch behavior.

The Ostwald ripening equation represents a low-dimensional nearly linear or weakly nonlinear system. Its state evolution is governed mainly by a few low-frequency modes, making it suitable for testing whether a local POD basis captures the backbone of the solution manifold and whether residual certification preserves physical branch consistency during parameter inversion.

The particle population-balance model represents a strongly nonlinear nonlocal system. It includes integral global feedback, nonnegativity constraints, and long-time evolution of a distribution function. Its solution manifold is more curved than that of OR and is more likely to produce local high-frequency oscillations and nonphysical negative values. This example tests the necessity of spectral decoupling and lightweight neural residual compensation.

The Bratu equation represents a parametric equilibrium solution manifold with a fold bifurcation. Near the critical parameter it exhibits Jacobian degeneracy and multiple solution branches. It tests the validity boundary of ordinary residual certification, the state-distance forgetting mechanism, and high-fidelity resampling.

In all experiments, high-fidelity models provide benchmark states. The reduced module extracts dominant POD modes. The residual compensation module uses a lightweight multilayer perceptron trained by physical residuals, boundary-condition losses, and regularization. Evaluation focuses on local manifold reconstruction, residual-triggered snapshot updates, stability of parameter inversion or branch tracking, empirical correspondence between residual and true error, and online cost as a function of parameter step and manifold complexity.

The experiments are designed to evaluate mechanisms rather than to optimize a single benchmark score. Each example emphasizes a distinct failure mode of standard surrogates. OR tests whether a local linear space is sufficient when the state manifold is smooth and low-dimensional. PMC tests whether a linear space remains adequate when nonlocal nonlinear feedback generates high-frequency residuals. Bratu tests whether residual certification can still guide local model management when the Jacobian becomes nearly singular. Across all examples, the same diagnostic quantities are recorded: residual norm, state error relative to high-fidelity reference, number and location of high-fidelity updates, reduced dimension, and online runtime. This common diagnostic structure is what makes the three examples comparable despite their different physical forms.

For the figures, the goal is not only to show final accuracy but to display the algorithmic logic. The OR and PMC figures illustrate problem-specific reconstruction and correction behavior. The composite figure reports cross-method comparisons, residual-control behavior, complexity regression, and the Bratu solution structure. Together, these four figures provide supporting evidence for the paper: the method compresses smooth low-dimensional dynamics, corrects nonlinear residuals in nonlocal systems, detects loss of validity near bifurcation, and concentrates computational cost at difficult parameter regions.

These numerical results should be read as mechanism-level evidence rather than as a claim of unconditional superiority over all reduced-order or neural solvers. The theoretical claims are local, and the experiments are used to check whether the observed residual, update, and correction behavior is consistent with those local claims.

\subsection{Ostwald ripening: local compression in low-dimensional dynamics}

Ostwald ripening describes the evolution of particle-size distributions in multiphase materials under interfacial energy and diffusion mechanisms; the classical mean-field description is the Lifshitz--Slyozov--Wagner theory \cite{lifshitz1961kinetics,wagner1961alterung}. Although its governing equation can be written as a parametric differential equation, the central task is to characterize how the state evolves along a physically compatible low-dimensional branch as the control parameter varies.

Let \(f(r,t;\mu)\) denote the particle-size distribution. A typical mean-field form is
\[
\partial_t f+\partial_r\bigl(v(r;\mu)f\bigr)=0,
\]
with appropriate conservation and normalization constraints. The parameter \(\mu\) controls growth rate, interfacial energy, or diffusion scale. The residual operator can be written as
\[
\mathcal R_{\mathrm{OR}}(f;\mu)
=\partial_t f+\partial_r(v(r;\mu)f).
\]
Snapshots are generated at different parameter values and time stages. POD extracts the principal distribution modes, and the reduced representation is
\[
f_{\mathrm{low}}(r,t;\mu)
=\bar f(r,t)+\sum_{j=1}^r a_j(\mu)\phi_j(r,t).
\]

The numerical results indicate that, under a two-level optimization architecture, the initial state converges toward a physical steady distribution with compact-support characteristics. The local POD basis captures the low-frequency backbone of the distribution, enabling online inference in a low-dimensional coordinate space. During parameter optimization, the outer objective displays a clear valley structure, indicating that physically compatible parameters lie in a stable minimum region. The trajectories of the total loss and intrinsic parameter \(\gamma\) suggest that the outer optimizer approaches the physical parameter along a relatively smooth descent direction.

The residual-triggering mechanism behaves as expected. In smooth parameter regions, the residual remains below the threshold, and the local POD basis remains valid over multiple steps. When the parameter path leaves the local neighborhood or distributional morphology changes significantly, the residual
\[
\eta_k=\|\mathcal R_{\mathrm{OR}}(f_{\mathrm{low}};\mu_k)\|
\]
increases and triggers high-fidelity resampling. The corresponding update nodes are a posteriori responses to local manifold deviation.

The OR example is also useful for interpreting the role of residual certification in an inverse or outer-optimization loop. In low-dimensional systems, one might expect POD alone to be sufficient. The experiment shows that POD indeed carries most of the approximation burden, but the residual still plays an important supervisory role. It prevents the optimizer from trusting a reduced state after the parameter path has moved outside the local snapshot region. In this regime, high-fidelity updates are infrequent, and the online cost is dominated by reduced solves rather than resampling. This is the favorable case for the proposed method: the residual certificate is rarely violated, but when it is violated, it prevents the reduced model from silently extrapolating.

Figure~\ref{fig:or} should therefore be read as evidence of local manifold coherence. Smooth changes in the reconstructed distributions indicate that the POD basis is aligned with the physical branch. Residual-triggered update points mark transitions where the old local affine space no longer represents the current distribution accurately enough. The absence of large oscillatory corrections in this example is consistent with the interpretation that OR is primarily a low-dimensional backbone problem rather than a high-frequency residual-compensation problem.

\begin{figure}[htbp]
\centering
\includegraphics[width=0.96\textwidth]{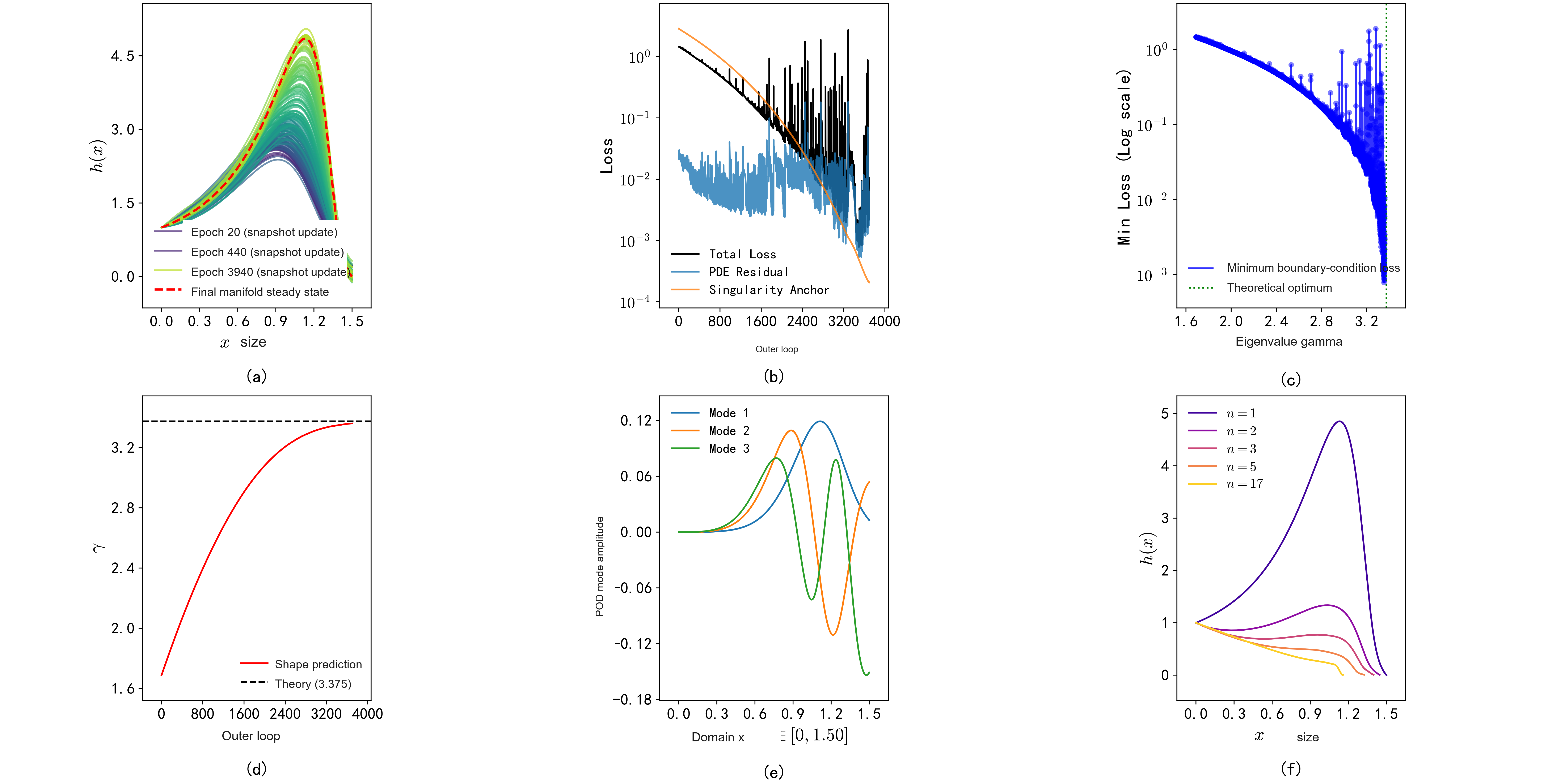}
\caption{Results for the Ostwald ripening example. The local POD model captures the dominant low-dimensional backbone of the distribution, while residual-triggered updates preserve consistency along the physical parameter branch.}
\label{fig:or}
\end{figure}

\subsection{Particle population balance: residual compensation in nonlinear nonlocal dynamics}

Particle population-balance models describe the evolution of particle distributions under coagulation, fragmentation, nucleation, and related mechanisms \cite{ramkrishna2000population,marchisio2005solution}. Compared with OR, the main difficulty is not only whether the low-dimensional manifold can be compressed; nonlocal integral coupling and strongly nonlinear feedback can significantly alter the local geometry of the solution manifold.

Let \(n(x,t;\mu)\) be a particle distribution. A representative model is
\[
\partial_t n(x,t)
=\mathcal Q_{\mathrm{coag}}[n](x)
+\mathcal Q_{\mathrm{frag}}[n](x)
+S(x;\mu),
\]
where \(\mathcal Q_{\mathrm{coag}}\) and \(\mathcal Q_{\mathrm{frag}}\) contain nonlocal integral terms. The residual is
\[
\mathcal R_{\mathrm{PMC}}(n;\mu)
=\partial_t n-\mathcal Q_{\mathrm{coag}}[n]
-\mathcal Q_{\mathrm{frag}}[n]-S(x;\mu).
\]
The nonlinearity and nonlocality make the residual sensitive to local high-frequency errors. A purely low-dimensional projection may reproduce the main distribution profile but still generate negative values, oscillations, or incorrect tail behavior.

The hybrid model writes
\[
n_{\mathrm{hyb}}=n_{\mathrm{low}}+g_\theta.
\]
Here \(n_{\mathrm{low}}\) is the POD reconstruction, and \(g_\theta\) compensates the residual component not captured by the linear subspace. This decoupling is important: the neural network does not learn the full population distribution. It learns only the missing high-frequency or nonlinear residual component after POD has captured the dominant structure.

The numerical results indicate that spectral decoupling and residual compensation significantly reduce nonphysical oscillations in this test. Compared with a pure POD reconstruction, the hybrid model better maintains nonnegativity and conservation-related quantities. Compared with a pure PINN, the network has a smaller hypothesis space and more stable optimization because the low-frequency backbone has already been represented by POD. The residual curves indicate that the compensator suppresses the accumulation of high-frequency error over the parameter path.

The PMC example highlights a different use of the residual. Here a small POD truncation error in the energy norm does not necessarily imply a small physical residual. Integral terms and nonlinear source terms can amplify errors that are visually small in the state profile. The neural correction is therefore evaluated not by its magnitude alone, but by its ability to reduce the physical residual while preserving shape constraints. A successful correction is expected to be modest in the \(H\)-norm but significant in the residual norm. This is precisely the regime in which residual-space learning is preferable to full-field neural approximation.

Figure~\ref{fig:pmc} illustrates this behavior. The low-dimensional component captures the main distributional profile, while the neural correction accounts for localized discrepancies induced by nonlocal feedback. The comparison with pure PINN and ROM in the composite figure further shows why the hybrid model is more stable: the pure PINN must learn all scales simultaneously, whereas the static ROM cannot adapt its basis when the nonlinear distributional structure changes. The hybrid model reduces both burdens by assigning different scales to different modules.

\begin{figure}[htbp]
\centering
\includegraphics[width=0.96\textwidth]{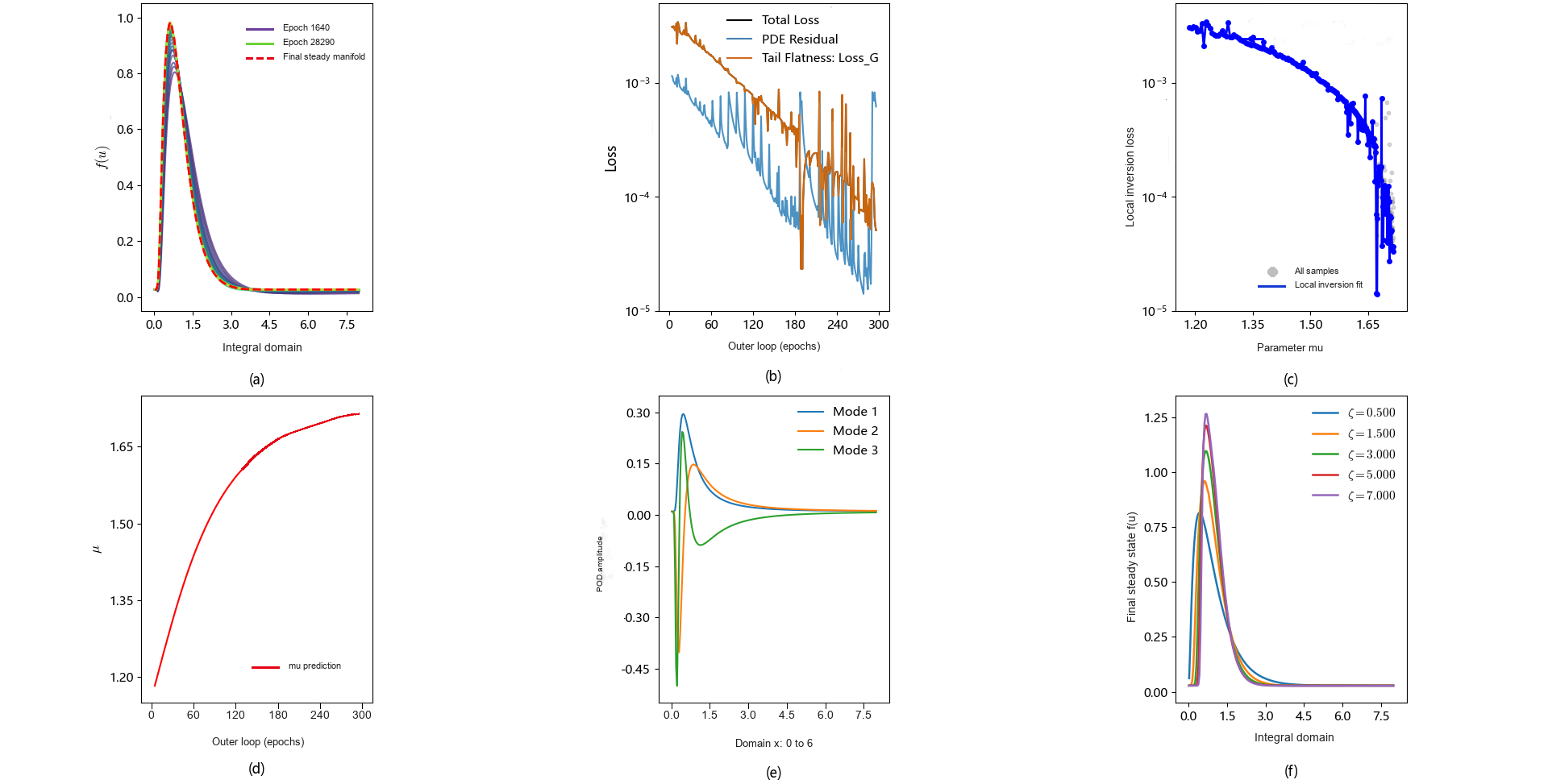}
\caption{Results for the particle population-balance model. The hybrid POD--neural approximation compensates nonlinear nonlocal residual components and suppresses high-frequency oscillations that are not captured by the reduced basis alone.}
\label{fig:pmc}
\end{figure}

\subsection{Bratu equation: residual boundary test near a fold}

The one-dimensional Bratu equation is a classical model for testing bifurcation-neighborhood residual behavior \cite{joseph1973quasilinear,keller1977numerical}:
\[
-u''(x)=\lambda e^{u(x)},\qquad x\in(0,1),
\]
with homogeneous Dirichlet conditions
\[
u(0)=u(1)=0.
\]
As \(\lambda\) approaches a critical value, the solution branch folds and the linearized operator
\[
D_u\mathcal R_{\mathrm{Bratu}}(u;\lambda)
\]
degenerates. Ordinary Newton iteration and ordinary linear residual--error equivalence may fail.

The Bratu experiment has a different theoretical role from OR and PMC. It does not simply test POD or neural compensation. It tests whether the residual indicator remains useful as an a posteriori signal for local model failure and snapshot update when the solution manifold enters a near-singular neighborhood. According to the preceding analysis, the ordinary residual
\[
\|\mathcal R_{\mathrm{Bratu}}(u;\lambda)\|_{\mathcal Y}
\]
cannot linearly control the state error near the fold. Under isolated-degeneracy and finite-branch separation assumptions, however, it may satisfy a fractional one-sided control:
\[
\operatorname{dist}_{\mathcal X}(u,\mathcal Z_\lambda)
\leq
C\|\mathcal R_{\mathrm{Bratu}}(u;\lambda)\|_{\mathcal Y}^{\alpha},
\qquad 0<\alpha\leq 1.
\]
Thus the experiment does not assume the linear residual--error equivalence at the critical point. It observes how residual growth, state-distance variation, and high-fidelity resampling are coupled. Near the fold, the residual threshold must be more conservative, and state-distance forgetting should avoid compressing snapshots from different local branches into the same POD subspace.

The results indicate that the algorithm advances stably along the regular part of the Bratu branch and triggers more frequent high-fidelity resampling and local basis updates near the fold. This should be interpreted as numerical evidence that the residual detects local model failure, not as evidence that the ordinary residual remains linearly equivalent to error at the singular point.

The Bratu experiment is intentionally not presented as a full branch-continuation method. The relevant question is narrower and aligned with the algorithm: when a reduced candidate is produced near a fold, can the residual reveal that the local basis is losing validity? The observed behavior supports an affirmative answer. In the regular part of the branch, the residual behaves approximately as a linear error indicator. Near the fold, the residual becomes more sensitive, update frequency increases, and the condition-number information warns that the regular-branch constants are deteriorating. The algorithm responds by reanchoring the local model rather than by assuming that the old residual--error constant remains valid.

This interpretation is important for consistency with the theoretical analysis. A fold destroys the local graph representation with respect to the original parameter, so a global continuation algorithm would require additional branch parametrization. The present work does not claim to solve that global problem. It claims that residual-based local validity checks remain useful up to their appropriate boundary, and that isolated degeneracy can be handled by stricter thresholds and high-fidelity updates when the goal is local active-branch tracking.

\subsection{Ablation and residual-control experiments}

To analyze the role of each module, the hybrid architecture is compared with pure PINN, static global ROM, and local POD without residual compensation. A pure PINN must learn the low-frequency backbone, high-frequency local structure, and nonlinear physical constraints simultaneously; it is susceptible to spectral bias and gradient competition. A static global ROM assumes that a fixed global subspace covers all parameter states. When the state manifold bends, folds, or changes topology locally, basis mismatch causes residual growth. It may also mix snapshots from different branches and create nonphysical intermediate states. Local POD without compensation mitigates some mismatch by updating the basis, but it can still miss high-frequency residuals in strongly nonlinear nonlocal systems such as PMC.

The proposed hybrid architecture distributes responsibilities: POD captures the low-frequency backbone; residual certification detects local manifold deviation; high-fidelity resampling reanchors the active branch; neural compensation corrects high-frequency residuals; and state-distance forgetting maintains branch consistency of the snapshot pool. The ablation experiments show more stable residual behavior for the hybrid model, supporting the effectiveness of this division of labor.

The significance of the ablation is not that every module is equally important in all problems. OR mainly relies on local POD; PMC relies more strongly on residual compensation; Bratu relies more strongly on residual-threshold tightening, state-distance forgetting, and high-fidelity resampling. This behavior is consistent with the progressive experimental design.

The ablation results also clarify why the proposed method should not be reduced to a single component. Removing residual certification makes the model cheaper but eliminates the mechanism that detects local basis failure. Removing state-distance forgetting preserves more data but increases the risk of mixing incompatible branch information. Removing neural compensation is acceptable in OR but harmful in PMC, where high-frequency residual components are structurally important. Replacing the hybrid model with a pure PINN increases flexibility but loses the stable low-dimensional backbone. These observations support the modular design: each component is activated by a different mathematical difficulty, and the relative importance of the components changes with the complexity of the dynamical system.

\subsection{Empirical verification of residual--error control}

On regular local branches, the theory gives
\[
c_1\|u-u_\alpha^\ast(\mu)\|_{\mathcal X}
\leq
\|\mathcal R(u;\mu)\|_{\mathcal Y}
\leq
c_2\|u-u_\alpha^\ast(\mu)\|_{\mathcal X}.
\]
To empirically assess this mechanism, test points are sampled in regular regions, and the computable residual
\[
\|\mathcal R(u;\mu)\|_{\mathcal Y}
\]
is compared with the state error relative to the high-fidelity benchmark,
\[
\|u-u_h\|_{\mathcal X}.
\]
The log--log scatter plot shows that most test points lie close to a linear envelope, supporting the use of residual thresholds for snapshot update in regular regions.

This empirical verification has clear boundaries. It should not be interpreted as a global error-equivalence proof over the full parameter domain. The theoretical equivalence holds only near regular branches. Near the Bratu bifurcation, the ordinary residual--error relation may degenerate to fractional one-sided control; those points should be marked separately and used to check whether residual exceedance triggers snapshot update and high-fidelity resampling.

In practice, the residual--error scatter plot is used as a calibration tool. On regular validation points, a near-linear log--log relationship supports the chosen residual threshold. If the slope or spread changes substantially, the threshold should be adjusted or the residual norm should be rescaled. Near bifurcation, the slope is expected to differ from one, and this difference is not a failure of the method; it is evidence that the problem has moved from the regular regime to a fractional-control regime. This is why the Bratu data should not be pooled indiscriminately with OR and PMC regular-region data when estimating residual--error constants.

\subsection{Complexity and parameter-step dependence}

The online computational cost as a function of parameter step tests the complexity decomposition. If the one-step residual growth satisfies
\[
\eta_{k+1}\leq \eta_k
+C_\mu\|\Delta\mu_k\|
+C_p\varepsilon_{\mathrm{POD}}
+C_\theta\varepsilon_\theta,
\]
then reducing the parameter step, improving POD truncation, or lowering compensation error lengthens the valid interval between high-fidelity updates. Consequently \(N_{\mathrm{upd}}\) decreases, and the online cost is reduced.

The regression curves show that the actual cost is jointly modulated by parameter step and local manifold complexity. In smooth regular regions, the reduced model can advance for a long distance and high-fidelity calls are rare. In high-curvature, strongly nonlinear, or bifurcating neighborhoods, the residual exceeds the threshold more easily, and high-fidelity update cost increases.

The experimental complexity is therefore summarized by
\[
T_{\mathrm{total}}
=K(T_{\mathrm{ROM}}+T_{\mathrm{NN}})
+N_{\mathrm{upd}}T_{\mathrm{HFM}}.
\]
This formula does not claim an exact scaling law. It explains that total cost is determined by reduced inference, neural compensation, and high-fidelity resampling, with expensive high-fidelity computation concentrated in local regions of rapid solution-manifold change.

The complexity regression in Figure~\ref{fig:composite} should be interpreted in this sense. The slope reflects the average cost of repeated reduced inference and local neural updates, while deviations from the fitted trend correspond to clusters of high-fidelity resampling. Smooth regions of the parameter path lie close to the reduced-cost baseline. Regions with stronger curvature, nonlinear residual growth, or bifurcation sensitivity produce upward deviations because \(N_{\mathrm{upd}}\) increases. This behavior is preferable to a uniformly expensive high-fidelity scan: the extra cost is paid precisely where the solution manifold requires new information.

The accuracy comparison in the same figure complements the complexity plot. A method that is cheap but inaccurate is not useful for certification, and a method that is accurate only by calling the high-fidelity solver at every step loses the advantage of reduction. The hybrid method is positioned between these extremes. It uses the reduced model and neural corrector when they are certified, and it calls the high-fidelity solver only when the residual indicates that the local surrogate has become unreliable.

\begin{figure}[htbp]
\centering
\begin{subfigure}[t]{0.48\textwidth}
\centering
\includegraphics[width=\textwidth]{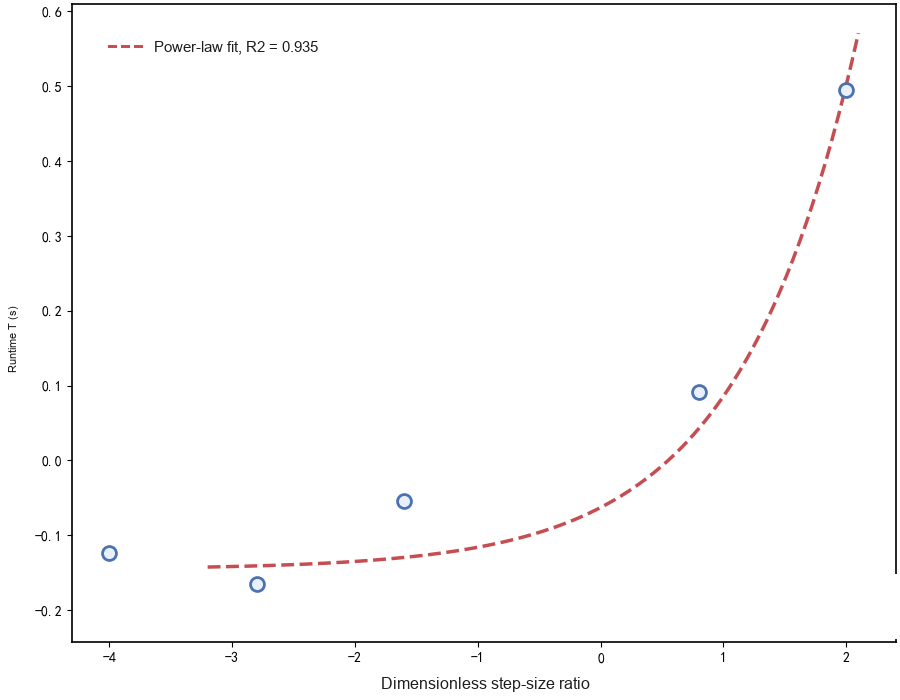}
\caption{Complexity regression.}
\end{subfigure}
\hfill
\begin{subfigure}[t]{0.48\textwidth}
\centering
\includegraphics[width=\textwidth]{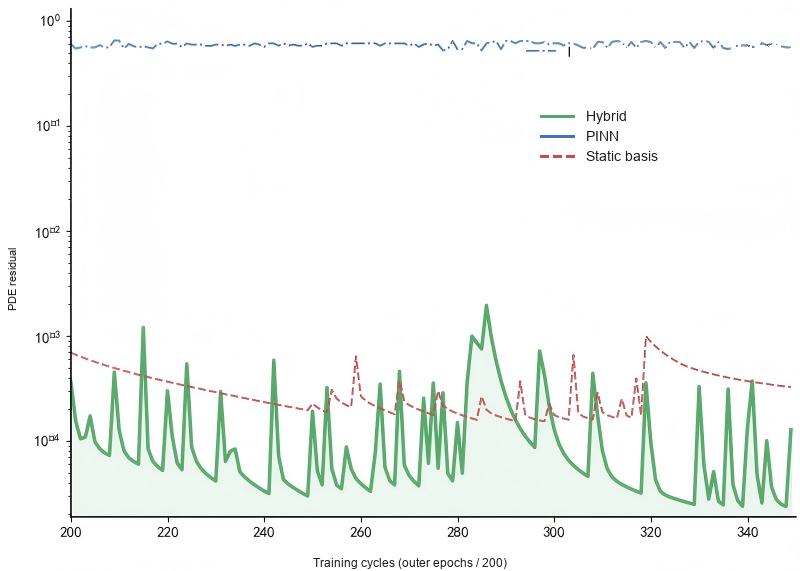}
\caption{Accuracy comparison with PINN and ROM.}
\end{subfigure}

\vspace{0.8em}
\begin{subfigure}[t]{0.48\textwidth}
\centering
\includegraphics[width=\textwidth]{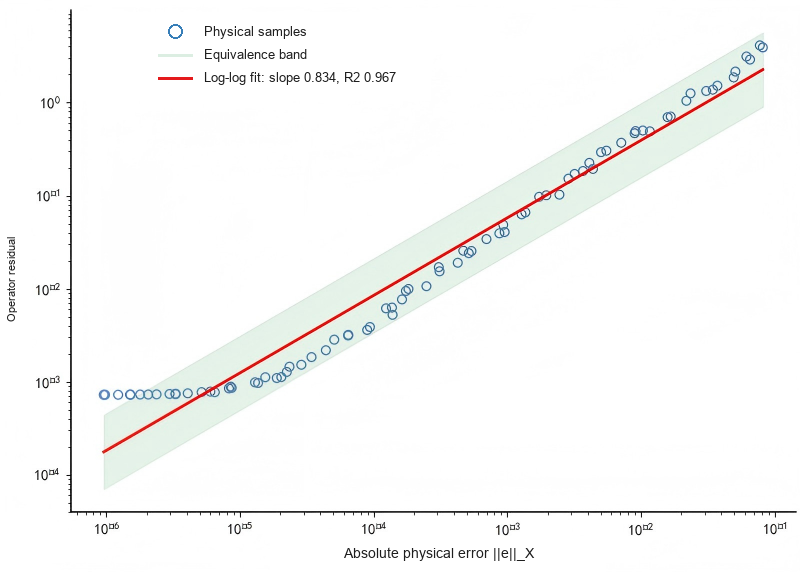}
\caption{Residual-control experiment.}
\end{subfigure}
\hfill
\begin{subfigure}[t]{0.48\textwidth}
\centering
\includegraphics[width=\textwidth]{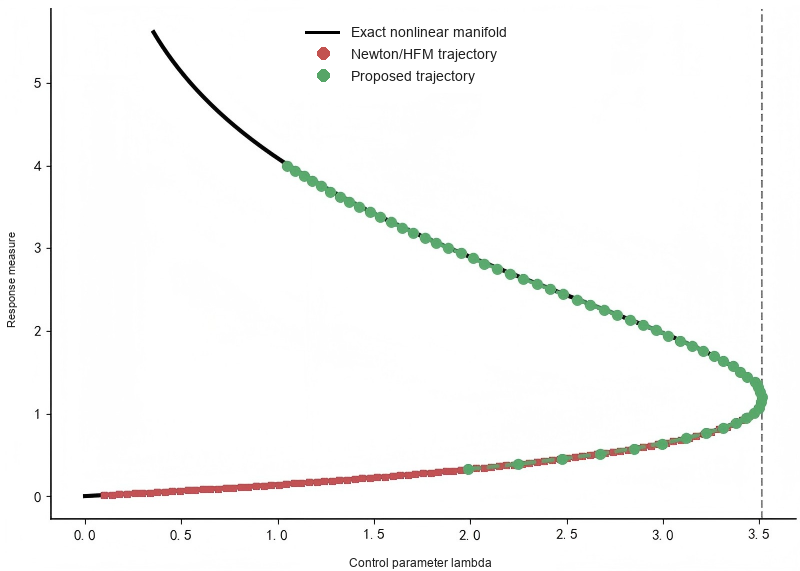}
\caption{Bratu solution behavior.}
\end{subfigure}
\caption{Composite numerical assessment. The four panels summarize the regression of time complexity, comparison against PINN and ROM baselines, residual-control behavior, and the Bratu equation solution experiment.}
\label{fig:composite}
\end{figure}

\subsection{Summary of numerical results}

The OR, PMC, and Bratu examples support the behavior expected from the proposed algorithm across different dynamical-system complexities. OR shows that local POD effectively compresses the solution-manifold backbone in low-dimensional or weakly nonlinear dynamics and that residual certification prevents the inverse process from leaving the physical branch. PMC shows that low-dimensional projection alone is insufficient for high-fidelity accuracy in strongly nonlinear nonlocal systems; spectral decoupling and neural residual compensation are important for suppressing nonphysical oscillations and high-frequency error accumulation. Bratu shows that ordinary linear residual--error equivalence fails near a fold because of Jacobian degeneracy, but the residual remains useful as an a posteriori indicator for local model failure and snapshot update.

The experiments further show that the algorithm's advantage comes from a layered design for parametric solution manifolds: local POD captures the low-frequency backbone, physical residuals certify credibility, state-distance forgetting maintains the active branch, neural networks compensate nonlinear residuals, and high-fidelity resampling reanchors the branch when the local model fails. The numerical results are consistent with the algorithmic design and mathematical analysis.

\section{Conclusion}

This paper developed a residual-certified adaptive hybrid method for tracking solution manifolds of parametric dynamical systems. The method represents evolutionary, equilibrium, and constrained models through a unified residual system, uses local POD to capture the low-dimensional backbone of the currently active branch, employs physical residuals as a posteriori certification indicators, updates snapshots through state-distance forgetting, and compensates nonlinear residual components with a lightweight physics-informed neural network. The central design principle is that no single surrogate is required to represent the full parametric solution set. Instead, the algorithm maintains a local, certified representation of the branch actually visited by the parameter path.

The theoretical analysis established the mathematical role of the residual in this framework. On regular solution branches, the implicit-function theorem and the local inverse-mapping theorem yield a two-sided equivalence between residual norm and state error. This equivalence justifies the use of residual thresholds for accepting reduced candidates, triggering high-fidelity resampling, and updating local bases. For multibranch systems, the equivalence must be interpreted relative to an active branch rather than a globally unique solution. This is why residual certification is coupled with state-distance forgetting: residual smallness certifies proximity to a solution set, while state-space locality helps preserve proximity to the intended physical branch.

The analysis also clarified the role of bifurcation neighborhoods. Near isolated folds or finite-branch pitchfork-type degeneracies, the ordinary linear residual--error equivalence may fail because \(D_u\mathcal R\) loses invertibility. Nevertheless, adaptive model management does not require a full global continuation method. It requires a one-sided estimate that controls the distance to the local solution set by a fractional power of the residual. Under local compactness, finite-branch structure, and branch-separation assumptions, such H\"older-type control is sufficient for deciding whether the current local surrogate remains reliable. This observation is important for the consistency of the present paper: because the experiments do not use pseudo-arclength continuation, the algorithmic core should not claim to rely on it. The necessary mechanism is residual-based detection of local model failure followed by high-fidelity reanchoring and snapshot purification, not complete continuation of all branches through a singular point.

The numerical experiments support this interpretation. In the Ostwald ripening example, the state manifold is largely low-dimensional and smooth, so local POD carries most of the approximation burden while the residual acts as a supervisory mechanism that prevents unsafe extrapolation. In the particle population-balance example, nonlocal nonlinear feedback produces residual components that are poorly represented by a linear basis; the neural compensator improves accuracy by learning only the residual part left after POD reconstruction. In the Bratu example, the residual becomes more sensitive near the fold, and high-fidelity updates occur more frequently. This behavior agrees with the theoretical prediction that the residual remains useful as a local failure indicator even though its relation to state error is no longer a regular linear equivalence.

The proposed architecture also gives a transparent decomposition of computational cost and error. The online cost consists of reduced solves, local neural updates, and high-fidelity resampling. Its efficiency comes from keeping the number of high-fidelity calls small in smooth regions while allowing those calls to concentrate near rapid manifold changes or near-degenerate points. The error decomposition separates high-fidelity discretization error, POD truncation error, local extrapolation error, neural compensation error, and residual-threshold error. This separation is useful for diagnosing which module limits accuracy in a given regime. For example, increasing the POD rank is effective when spectral truncation dominates, but it cannot fix geometric extrapolation outside the snapshot neighborhood; increasing neural-network capacity helps in PMC-type residual compensation, but it is not the right response when the POD candidate has already failed residual certification.

Several limitations remain. The theory is local and does not provide unconditional convergence over arbitrary parameter domains. If singular points form a continuum, if branches remain indistinguishable in the state norm over an extended interval, or if the high-fidelity solver itself switches branches unpredictably, the proposed branch-consistency mechanism may fail. The constants in the residual--error estimates may also become large for highly nonnormal or severely ill-conditioned operators, making residual thresholds conservative. In addition, the present implementation uses relatively simple snapshot forgetting and lightweight neural correction; more sophisticated active sampling, uncertainty quantification, or structure-preserving neural architectures may improve robustness.

Future work can proceed in several directions. First, the residual-certified update rule can be combined with probabilistic error indicators to distinguish deterministic residual growth from noise or discretization artifacts. Second, the snapshot management strategy can be extended to multiple simultaneously active branches by clustering in state space and maintaining several local POD bases in parallel. Third, for high-dimensional multiphysics systems, hyper-reduction and operator compression will be needed to reduce the cost of full residual evaluation. Fourth, the fractional residual-control analysis near bifurcation can be developed into computable local threshold rules, where the exponent \(\alpha\) is estimated from residual--error data. These extensions would preserve the main principle advanced in this paper: adaptive computation should be organized around certified local validity of the active solution manifold, with high-fidelity effort deployed only when the residual shows that the current local model has reached its boundary of reliability.

\end{document}